\documentclass[reqno, 11pt]{amsart}
\usepackage{amsmath, amsthm, amscd, amsfonts, amssymb, graphicx, color, mathrsfs}
\usepackage{extarrows}
\usepackage[bookmarksnumbered, colorlinks, plainpages]{hyperref}
\hypersetup{colorlinks=true,linkcolor=red, anchorcolor=green,
	citecolor=cyan, urlcolor=red, filecolor=magenta, pdftoolbar=true}

\newtheorem{theorem}{Theorem}[section]
\newtheorem{lemma}[theorem]{Lemma}
\newtheorem{proposition}[theorem]{Proposition}
\newtheorem{corollary}[theorem]{Corollary}
\theoremstyle{definition}

\newtheorem{example}[theorem]{Example}

\newtheorem{question}[theorem]{Question}
\theoremstyle{remark}
\newtheorem{remark}[theorem]{Remark}
\numberwithin{equation}{section}

\begin{document}
	
	\setcounter{page}{1}
	
	\title[Power quasinormal operators and the root problem]
	{Power quasinormal operators and the root problem}

	\author[J.B. Zhou]{Jing-Bin Zhou}
	\address{Jing-Bin Zhou, School of Mathematics,
		Shanghai University of Finance and Economics,
		777 Gouding Road, Shanghai 200433, P. R. China}
	\email{zhoujingbin@stu.sufe.edu.cn}
	
	\author[S. Yang]{Shihai Yang}
	\address{Shihai Yang, School of Mathematics,
		Shanghai University of Finance and Economics,
		777 Gouding Road, Shanghai 200433, P. R. China}
	\email{yang.shihai@mail.shufe.edu.cn}

	
	\begin{abstract}
		
			In this paper, we construct an $n$-power quasinormal operator $T$ such that $T^n$ is not quasinormal for some positive integer $n$, thereby providing a counterexample to \cite[Lemma 3.1]{ko-filomat-2023}. We then investigate the relationships among the $n$-power quasinormality of $T$, the normality of $T^n$, and the quasinormality of $T^n$, and show that these three conditions are equivalent in finite-dimensional spaces. We also provide a new proof that $n$-power quasinormal operators have the single-valued extension property \cite[Theorem 3.2]{ko-filomat-2023}; unlike the original proof, our argument does not rely on \cite[Lemma 3.1]{ko-filomat-2023} and thus closes the gap in the original argument. In addition, for a fixed operator $T$, we characterize all positive integers $n$ for which $T$ is $n$-power quasinormal. As consequences, several results of Sid Ahmed \cite{ahmed-bmaa-2011} are extended. Finally, we prove that every paranormal $n$-power quasinormal operator is quasinormal.  Closely related to this, we also give an affirmative answer to the root problem of Stankovi\'c and Kubrusly \cite[Question 2.11]{stankovic-afa-2025}.
	\end{abstract}
	\maketitle
	\section{Introduction}
	
		Let \(\mathbb{B}(\mathcal H)\) denote the algebra of all bounded linear operators on a complex Hilbert space \((\mathcal H, \langle \cdot, \cdot \rangle)\). We denote the adjoint, norm, null space, and range of $T\in\mathbb{B}(\mathcal{H})$ by $T^\ast$, $\|T\|$, $\ker T$ and $\operatorname{ran} T$, respectively. For any \(T, S \in \mathbb B(\mathcal H)\), we write $[S,T]=ST-TS$ for the commutator. An operator \(T\) is called \emph{normal} if \([T, T^*] = 0\). Largely owing to the development of the Spectral Theorem, normal operators have been studied extensively and systematically, and various generalizations of normal operators have become one of the main directions in operator theory. Specifically, Brown \cite{brown} called an operator \(T\) \emph{quasinormal} if \([T, T^*T] = 0\). A closed subspace \(M\) of \(\mathcal H\) is said to be \emph{invariant} under \(T\) if \(TM \subseteq M\). If both \(M\) and its orthogonal complement \(M^\perp\) are invariant under \(T\), then \(M\) is called a \emph{reducing subspace} for \(T\). Halmos \cite{Halmos-GTM-1982} called \(T\) \emph{subnormal} if \(T \in \mathbb B(\mathcal H)\) is the restriction of a normal operator \(N\) on some Hilbert space \(\mathcal K \supseteq \mathcal H\) to an invariant subspace \(\mathcal H\), i.e., \(N|_{\mathcal H} = T\). If \(S, T \in \mathbb B(\mathcal H)\) and \(\langle (S - T)x, x \rangle \ge 0\) for all \(x \in \mathcal H\), we write \(S \ge T\). Halmos \cite{Halmos-GTM-1982} called \(T\) \emph{hyponormal} if \(TT^* \le T^*T\). If \(T \ge 0\), then \(T\) is called \emph{positive}. The square root of the positive operator \(T^*T\) is denoted by \(|T|\). Furuta, Ito and Yamazaki \cite{furuta} called \(T\) \emph{of class A} if \(|T|^2 \le |T^2|\). Istrăţescu \cite{is-pjm-1967} called \(T\) \emph{paranormal} if \(\|Tx\|^2 \le \|T^2x\| \|x\|\) for all \(x \in \mathcal H\). It is well known that these generalized normal operators satisfy the following strict inclusions:
		\begin{equation}\label{normalclass}
			\text{normal} \subsetneq \text{quasinormal} \subsetneq \text{subnormal} \subsetneq \text{hyponormal} \subsetneq \text{class A} \subsetneq \text{paranormal}.
		\end{equation}
		
		Besides the above generalized normal operators, \(n\)-power quasinormal operators also form a generalization of normal operators, and this is the main object of study in this paper. Let \(\mathbb N\) denote the set of positive integers. Sid Ahmed \cite{ahmed-bmaa-2011} called an operator \(T\) \emph{\(n\)-power quasinormal} if \([T^n, T^*T] = 0\) for some \(n \in \mathbb N\). A closely related class is that of \(n\)-power normal operators, introduced by Jibril \cite{jibril-jse-2008}: an operator \(T\) is called \emph{\(n\)-power normal} if \([T^n, T^*] = 0\) for some \(n \in \mathbb N\). Jibril proved that \(T\) is \(n\)-power normal if and only if \(T^n\) is normal \cite[Proposition 2.1]{jibril-jse-2008}. Furtado and Johnson \cite[Theorem 5.2]{fur-lma-2022} gave the structure of \(n\)-power normal matrices in finite-dimensional spaces.
		
		Recently, in the theory of generalized normal operators, the following problem has attracted renewed attention. Let $\mathscr{A}$ be a class of operators and $\mathscr{B}$ a subclass of $\mathscr{A}$. For $n\in\mathbb{N}$, if $T\in\mathscr{A}$ and $T^n\in\mathscr{B}$, does it follow that $T\in\mathscr{B}$? This problem is known as {\em the $n$-th root problem (for the class $\mathscr{B}$)} \cite{stankovic-afa-2025}. In this direction, Stampfi first proved the following result.
\begin{theorem}\cite[Theorem 5]{sta-pjm-1962}
If $T$ is hyponormal and $T^n$ is normal for some $n\in\mathbb{N}$, then $T$ is normal.
\end{theorem}
 Subsequently, Ando extended this result.
\begin{theorem}\cite[Theorem 6]{ando-act-1972}
If $T$ is paranormal and $T^n$ is normal for some $n\in\mathbb{N}$, then $T$ is normal.
\end{theorem} 
Curto et al. recently initiated the study of the $n$-th root problem for quasinormal operators.
\begin{theorem}\cite[Theorem 2.3]{curto-jfa-2020}
If $T$ is a left-invertible subnormal operator and $T^2$ is quasinormal, then $T$ is quasinormal.
\end{theorem}
Pietrzycki and Stochel removed the left invertibility assumption in the above result and obtained a more general version.
\begin{theorem}\cite[Theorem 1.2]{sto-jfa-2021}
If $T$ is subnormal and $T^n$ is quasinormal for some $n\in\mathbb{N}$, then $T$ is quasinormal.
\end{theorem}
Later, the same authors extended the above result and established the following.

\begin{theorem}\cite[Theorem 1.2]{sto-ann-2023}
	If $T$ is of class A and $T^n$ is quasinormal for some $n\in\mathbb{N}$, then $T$ is quasinormal.
\end{theorem}
Motivated by \eqref{normalclass} and the preceding results, Stanković and Kubrusly naturally raised the following question in \cite[Question 2.11]{stankovic-afa-2025}.
\begin{question}\label{question1}
If $T$ is paranormal and $T^n$ is quasinormal for some $n\in\mathbb{N}$, does it follow that $T$ is quasinormal?
\end{question}

We also pose the following question.

\begin{question}\label{question2}
	If \(T\) is \(n\)-power quasinormal and \(T^n\) is quasinormal for some \(n \in \mathbb N\), does it follow that \(T\) is quasinormal?
\end{question}

Note that Ko and Lee recently used the following result in their proof of the single-valued extension property for \(n\)-power quasinormal operators on Hilbert spaces of arbitrary dimension.

\begin{theorem}\cite[Lemma 3.1]{ko-filomat-2023}\label{kolee}
	If $T\in\mathbb{B}(\mathcal{H})$ is $n$-power quasinormal, then $T^n$ is quasinormal. Conversely, if $T^n$ is quasinormal and $\ker (T^\ast)^n\subseteq \ker T^n$, then $T$ is $n$-power quasinormal.
\end{theorem}	

In view of this theorem, the condition \(T^n\) is quasinormal in Question \ref{question2} appears to be superfluous. In fact, however, we give an example of an \(n\)-power quasinormal operator \(T\) in infinite dimensions for which \(T^n\) is not quasinormal (Example \ref{pnnotimplynq}). This motivates us to study the relationship between an \(n\)-power quasinormal operator \(T\) and the quasinormal operator \(T^n\). In particular, in Theorem \ref{mainthm} we prove that, in finite dimensions, the following are equivalent: \(T\) is \(n\)-power quasinormal; \(T^n\) is quasinormal; \(T^n\) is normal; \(T\) is \(n\)-power normal. It is also well known that quasinormal matrices and normal matrices coincide in finite dimensions. Therefore, by taking a non-normal nilpotent matrix, we obtain a negative answer to Question \ref{question2}. From the finite-dimensional perspective, another natural question arises by replacing ``\(T^n\) is quasinormal'' in Question \ref{question1} with ``\(T\) is \(n\)-power quasinormal'', leading to the following question.

\begin{question}\label{question3}
	If \(T\) is paranormal and \(T\) is \(n\)-power quasinormal for some \(n \in \mathbb N\), does it follow that \(T\) is quasinormal?
\end{question}

Our Theorem \ref{main1} gives an affirmative answer to the above question. Using this result, we also characterize when \(T\) is quasinormal under the conditions of Question \ref{question2} (Corollary \ref{powerquasi}). Stanković and Kubrusly's Question \ref{question1} also arises naturally under Corollary \ref{powerquasi}; we give an affirmative answer to their question at the end of the paper.

The organization of this paper is as follows. In Section 2, we give a counterexample to the first part of Theorem \ref{kolee}, characterize when an \(n\)-power quasinormal operator is \(n\)-power normal (Theorem \ref{coromain}), and use this result to establish the relationship (Theorem \ref{mainthm}) between \(n\)-power quasinormal operators \(T\), quasinormal operators \(T^n\), and \(n\)-power normal operators \(T\), thereby giving the correct formulation of Theorem \ref{kolee}. Some known results are also strengthened. At the end of this section, we give a new proof of Theorem \ref{svep} \cite[Theorem 3.2]{ko-filomat-2023}, which avoids the use of \cite[Lemma 3.1]{ko-filomat-2023} and thereby removes the gap in the original argument.

In Section 3, using the matrix representation (Proposition \ref{MR}) of \(n\)-power quasinormal operators from Section 2, we characterize all positive integers \(n\) for which \(T\) is \(n\)-power quasinormal. This is of independent interest. As corollaries, we extend several results of Sid Ahmed \cite{ahmed-bmaa-2011}.

In the final section, we give a characterization of \(n\)-power quasinormal operators using the \(\lambda\)-Aluthge transform. As a corollary, we obtain a clear necessary and sufficient condition for an \(n\)-power quasinormal operator \(T\) to have \(T^n\) quasinormal. Furthermore, involving the $\lambda$-Aluthge transform, we characterize when \(T\) is quasinormal under the conditions of Question \ref{question2}. Questions \ref{question1} and \ref{question3} are also answered affirmatively. Some examples are provided to show that our results are sharp.

	\section{Relations among $n$-power quasinormality, $n$-power normality and quasinormality of powers}
We begin this part with a counterexample to the first part of Theorem \ref{kolee}. Let $S=\{s_{1},s_{2},s_{3},\cdots\}\subseteq\mathcal{H}$. Denote by $P_{\langle s_{1},s_{2},s_{3},\cdots\rangle}$ the orthogonal projection of $\mathcal{H}$ onto the closed linear span of $S$. 
 Denote by $\mathbb{N}_{0}$ the set of nonnegative integers, and by $l^2(\mathbb{N}_{0})$ the Hilbert space of square‑summable sequences, with orthonormal basis $\{e_{k}\}_{k\in\mathbb{N}_{0}}$. Let $\alpha=\{\alpha_{0},\alpha_{1},\alpha_{2},\cdots\}$ be a bounded sequence of positive real numbers. The weighted shift operator $W_{\alpha}$ on $l^2(\mathbb{N}_{0})$ with weight sequence $\alpha$ is defined by $W_{\alpha}e_{k}=\alpha_{k}e_{k+1}$ for $k\in\mathbb{N}_{0}$. The notation $\operatorname{shift}(\alpha_{0},\alpha_{1},\alpha_{2},\cdots)$ is also used for this operator, as it directly indicates the weights \cite[Problem 89]{Halmos-GTM-1982}.
 \begin{example}\label{pnnotimplynq}
 	Let $n\ge2$, and let $\mathcal{K}=l^2(\mathbb N_0)$. Fix $0<c<1$ and set $W=\operatorname{shift}(c,1,1,1,\cdots)$. Define the operator
 	\begin{equation}\label{poma3}
 		T=\begin{bmatrix}
 			W&0\\\sqrt{1-c^2}P_{\langle e_{0}\rangle}&0
 		\end{bmatrix}
 	\end{equation}
 	on $\mathcal{H}=\mathcal{K}\oplus\mathcal{K}$. Then $T$ is $n$-power quasinormal, but $T^n$ is not quasinormal. 
 	
 	Indeed, since $W^\ast e_{k}=\begin{cases}
 		0&k=0\\ce_{k-1}&k=1\\e_{k-1}&k\ge2
 	\end{cases}$, we have $W^\ast W+(1-c^2)P_{\langle e_{0}\rangle}=I$, and hence 
 	\begin{equation}\label{aluthgema}
 		T^\ast T=	\begin{bmatrix}
 			I&0\\
 			0&0
 		\end{bmatrix}.
 	\end{equation}
 Moreover, $P_{\langle e_{0}\rangle}W=0$, so 
 	\begin{equation}\label{poma}
 		T^n=
 		\begin{bmatrix}
 			W^n&0\\
 			0&0
 		\end{bmatrix}.
 	\end{equation} It follows immediately that $
 	T^n(T^*T)=(T^*T)T^n$; hence $T$ is $n$-power quasinormal. On the other hand, a straightforward induction on $n$ gives
 	\begin{equation}\label{poma1}
 		W^ne_k=\begin{cases}
 			ce_{n}&k=0\\e_{k+n}&k\ge1
 		\end{cases},\quad\text{and}\quad (W^\ast)^nW^ne_{k}=\begin{cases}
 			c^2e_{0}&k=0\\e_{k}&k\ge1
 		\end{cases}.
 	\end{equation}
 	Therefore, we obtain $$T^n(T^\ast)^nT^n\begin{pmatrix}
 		e_{0}\\0
 	\end{pmatrix}=\begin{pmatrix}
 		W^n(W^\ast)^nW^ne_{0}\\0
 	\end{pmatrix}=\begin{pmatrix}
 		c^3e_{n}\\0
 	\end{pmatrix},$$
 	and
 	$$(T^\ast)^nT^nT^n\begin{pmatrix}
 		e_{0}\\0
 	\end{pmatrix}=\begin{pmatrix}
 		(W^\ast)^nW^nW^ne_{0}\\0
 	\end{pmatrix}=\begin{pmatrix}
 		ce_{n}\\0
 	\end{pmatrix}.$$
 	Since $0<c<1$, we have $c^3\neq c$, and thus $T^n(T^\ast)^nT^n\ne (T^\ast)^nT^nT^n$. Consequently, $T^n$ is not quasinormal. 
 		The proof is complete.
 \end{example}
 
 	The next example is a slight modification of the preceding one. It shows that the condition $\ker(T^\ast)^n\subseteq\ker T^n$ in the latter part of Ko and Lee’s result is not merely technical. Furthermore, the fact that  $T^n$ is quasinormal for all $n\ge2$ does not generally imply that $T$ is $m$-power quasinormal for some $m\in\mathbb{N}$. 
 
 \begin{example}\label{pqsetempty}
 	Let $\mathcal{K}=l^2(\mathbb N_0)$. Define the operator $T=\begin{bmatrix}
 		W&0\\P_{\langle e_{0}\rangle}&0
 	\end{bmatrix}$ on $\mathcal{H}=\mathcal{K}\oplus\mathcal{K}$, where $W=\operatorname{shift}(1,1,1,\cdots)$. Then $T^n$ is quasinormal for every $n\ge2$, with $\ker (T^\ast)^n\not\subseteq\ker T^n$, but $T$ is not $m$-power quasinormal for any $m\in\mathbb{N}$.
 	
 	Indeed, since \begin{equation}\label{poma5}
 		P_{\langle e_{0}\rangle}W=0,
 	\end{equation}
 	we have $T^2=\begin{bmatrix}
 		W^2&0\\0&0
 	\end{bmatrix}$. A straightforward induction gives, for $n\ge2$,
 	\begin{equation}\label{poma4}
 		T^n=\begin{bmatrix}
 			W^n&0\\
 			0&0
 		\end{bmatrix}.
 	\end{equation}
 	Since $W^\ast W=I$, we get $(W^n)^\ast W^n=I$ and hence $(T^n)^\ast T^n=\begin{bmatrix}
 		I&0\\0&0
 	\end{bmatrix}$. Thus $T^n$ is quasinormal for every $n\ge2$. Next, set $\xi=\begin{pmatrix}
 		e_{0}\\0
 	\end{pmatrix}$. From \eqref{poma4}, a direct computation shows that $\xi\in\ker (T^\ast)^n$ but $\xi\not\in\ker T^n$. Therefore $\ker (T^\ast)^n\not\subseteq\ker T^n$ for all $n\ge2$. To see that $T$ is not $m$-power quasinormal for any $m\in\mathbb{N}$, note that $T^\ast T=
 	\begin{bmatrix}
 		I+P_{\langle e_{0}\rangle} & 0 \\
 		0 & 0
 	\end{bmatrix}$.
 	A direct calculation using \eqref{poma5} gives
 	$
 	[T,T^\ast T]=
 	\begin{bmatrix}
 		W P_{\langle e_{0}\rangle} & 0 \\
 		2P_{\langle e_{0}\rangle} & 0
 	\end{bmatrix}
 	\ne 0$. Moreover, for \(m\ge2\), using \eqref{poma5} and \eqref{poma4}, we have
 	$$
 	[T^m,T^\ast T]
 	\begin{pmatrix} e_0 \\ 0 \end{pmatrix}
 	=\begin{bmatrix}
 		W^mP_{\langle e_{0}\rangle}-P_{\langle e_{0}\rangle}W^m&0\\0&0
 	\end{bmatrix}\begin{pmatrix}
 		e_{0}\\0
 	\end{pmatrix}=\begin{pmatrix} e_m \\ 0 \end{pmatrix}
 	\ne 0.$$
 	Thus we get the desired result.
 \end{example}

 The following proposition for weighted shift operators also shows that the condition $\ker (T^\ast)^n\subseteq\ker T^n$ in Theorem \ref{kolee} is not necessary.
 
 \begin{proposition}
 	Let $n\in\mathbb{N}$ and $T=\operatorname{shift}(\alpha_{0},\alpha_{1},\alpha_{2},\cdots)$. Then $T$ is $n$-power quasinormal if and only if $T^n$ is quasinormal.
 \end{proposition}
 \begin{proof}
 	Note that $\ker (T^\ast)^n\not\subseteq\ker T^n$. Indeed, a computation similar to that in Example \ref{pnnotimplynq} yields $$\ker (T^\ast)^n=\Bigg\{\sum\limits_{k=0}^{n-1}c_{k}e_{k}:c_{k}\in\mathbb{C}\Bigg\}\quad\text{and}\quad\ker T^n=\{0\}.$$
 	By \cite[Proposition 3.7]{ko-filomat-2023}, $T$ is
 	$n$-power quasinormal if and only if
 	\begin{equation}\label{pq1}
 		\alpha_{k+n}=\alpha_{k},
 		\quad k\geq0,
 	\end{equation}
 	and $T^n$ is quasinormal if and only if
 	\begin{equation}\label{pq2}
 		\prod_{j=0}^{n-1}\alpha_{k+j}
 		=
 		\prod_{j=0}^{n-1}\alpha_{k+n+j},
 		\quad k\geq0.
 	\end{equation}
 	
 	\textit{Necessity}. Suppose that $T$ is $n$-power quasinormal. The desired result is immediate from \eqref{pq1} and \eqref{pq2}. 
 	
 	\textit{Sufficiency}. Suppose that $T^n$ is quasinormal. For $k\ge0$, define $r_k=\frac{\alpha_{k+n}}{\alpha_k}$. Since $\alpha_{k}>0$, it suffices to prove $r_k=1$ for all $k\ge0$. Indeed, \eqref{pq2} is equivalent to
 	\begin{equation}\label{pq4}
 		\prod_{j=0}^{n-1}r_{k+j}=1,
 		\qquad k\geq0.
 	\end{equation}
 	Hence, $$r_kr_{k+1}\cdots r_{k+n-1}=1\quad\text{and}\quad r_{k+1}r_{k+2}\cdots r_{k+n}=1,$$
 	from which we obtain 
 	\begin{equation}\label{pq3}
 		r_{k+n}=r_k.
 	\end{equation}
 	Now, for every $k\geq0$ and $m\ge0$, by the above equation we have
 	$$\alpha_{k+mn}=\frac{\alpha_{k+mn}}{\alpha_{k+(m-1)n}}\frac{\alpha_{k+(m-1)n}}{\alpha_{k+(m-2)n}}\cdots\frac{\alpha_{k+n}}{\alpha_{k}}\alpha_{k}=\alpha_{k}\prod_{j=0}^{m-1}r_{k+jn}=\alpha_{k}r_{k}^m.$$
 	Since $T$ is bounded, its weight sequence $\{\alpha_{k}\}_{k\ge0}$ is bounded. Therefore,
 	\begin{equation}\label{pq5}
 		r_k\leq1。
 	\end{equation}
 	On the other hand, taking $k=0$ in \eqref{pq4} yields $r_0r_1\cdots r_{n-1}=1$. It follows from \eqref{pq5} that $r_0=r_1=\cdots=r_{n-1}=1$. Then by \eqref{pq3}, we conclude that $r_{k}=1$ for $k\ge0$, as desired.
 \end{proof}
 
Note that in finite-dimensional spaces, $T^n$ is normal if and only if $T^n$ is quasinormal. To better understand the relationship between the quasinormal operator $T^n$ and the $n$-power quasinormal operator $T$, we first prove that $T^n$ is normal if and only if $T$ is $n$-power quasinormal and $\ker (T^\ast)^n\subseteq\ker T^n$ (Theorem \ref{coromain}). For this, we need the following operator matrix form. Denote by $\overline{A}$ the closure of a subset $A$ in a metric space.
	\begin{proposition}\label{MR}
		Let $n\in\mathbb{N}$ and $T\in\mathbb{B}(\mathcal{H})$. If $T$ is $n$-power quasinormal, then $T$ has the matrix representation 
		\begin{equation}\label{mr5}
			T=\begin{bmatrix}
				\widetilde{T}&X\\0&N
			\end{bmatrix}
		\end{equation}
		with respect to the space decomposition $\mathcal{H}=\overline{\operatorname{ran}T^n}\oplus\ker (T^\ast)^n$, where $\widetilde{T}$ is an injective $n$-power quasinormal operator satisfying $\widetilde{T}^\ast X=0$, and $N$ is nilpotent with $N^n=0$.
	\end{proposition}	
	\begin{proof}
		Since $T(\operatorname{ran}T^n)\subseteq\operatorname{ran}T^n$ and $T$ is bounded, the subspace $\overline{\operatorname{ran}T^n}$ is invariant under $T$. Using $\overline{\operatorname{ran}T^n}^\perp=\ker (T^\ast)^n$, we obtain the block form
		\begin{equation}\label{mr3}
			T=
			\begin{bmatrix}
				\widetilde{T} & X\\
				0 & N
			\end{bmatrix}
		\end{equation}
		with respect to the decomposition $\mathcal{H}=\overline{\operatorname{ran}T^n}\oplus\ker (T^\ast)^n$. Consequently,
		\begin{equation}\label{mr}
			T^\ast T
			=
			\begin{bmatrix}
				\widetilde{T}^\ast\widetilde{T}
				&
				\widetilde{T}^\ast X
				\\[1mm]
				X^\ast\widetilde{T}
				&
				X^\ast X+N^\ast N
			\end{bmatrix}.
		\end{equation}
		Now, for any $y=T^nx\in\operatorname{ran}T^n$, we have $$ (T^\ast T)y
		=
		(T^\ast T)T^nx
		=
		T^n(T^\ast T)x
		\in
		\operatorname{ran}T^n,$$
		where the second equality holds since $T$ is $n$-power quasinormal. Hence, $\overline{\operatorname{ran}T^n}$ is invariant under $T^\ast T$, which implies the off-diagonal entries 
		\begin{equation}\label{mr1}
			\widetilde{T}^\ast X=X^\ast\widetilde{T}=0
		\end{equation}
		in \eqref{mr}. On the other hand, a straightforward induction gives
		\begin{equation}\label{mr2}
			T^k=\begin{bmatrix}
				\widetilde{T}^k&\sum\limits_{i=0}^{k-1}\widetilde{T}^{k-1-i}XN^{i}\\0&N^k
			\end{bmatrix}
		\end{equation} 
		for all $k\in\mathbb{N}$. Then we get $N^n=P_{\overline{\operatorname{ran}T^n}^\perp}T^n|_{\ker (T^\ast)^n}=0$. From \eqref{mr}, \eqref{mr1} and \eqref{mr2}, the subspace $\overline{\operatorname{ran}T^n}$ is invariant under both $T^\ast T$ and $T^n$. Therefore,
		$$(T^\ast T)|_{\overline{\operatorname{ran}T^n}}=\widetilde{T}^\ast\widetilde{T},\quad T^n|_{\overline{\operatorname{ran}T^n}}=\widetilde{T}^n.$$ Restricting the equality $T^n(T^\ast T)=(T^\ast T)T^n$ to $\overline{\operatorname{ran}T^n}$ yields $
		\widetilde{T}^{\,n}
		\bigl(\widetilde{T}^\ast \widetilde{T}\bigr)
		=
		\bigl(\widetilde{T}^\ast \widetilde{T}\bigr)
		\widetilde{T}^{\,n}$, so $\widetilde{T}$ is $n$-power quasinormal. It remains to show that $\widetilde{T}$ is injective. By $[T^n,T^\ast T]=0$ and \cite[Proposition 2.6]{xu-aot-2018}, we have $[T^n,P_{\overline{\operatorname{ran}(|T|)}}]=0$. Hence,  $$T^n=T^n(P_{\overline{\operatorname{ran}(|T|)}}+P_{\ker |T|})=T^n(P_{\overline{\operatorname{ran}(|T|)}}+P_{\ker T})=P_{\overline{\operatorname{ran}(|T|)}}T^n,$$ which gives $$\overline{\operatorname{ran}T^n}\subseteq\overline{\operatorname{ran}(|T|)}=(\ker|T|)^\perp=(\ker T)^\perp.$$ Suppose that $x\in\overline{\operatorname{ran}T^n}\subseteq(\ker T)^\perp$ and $\widetilde{T}x=0$. By \eqref{mr3}, we have $Tx=0$, and
		thus $x\in\ker T$. It follows that $
		x\in\ker T\cap(\ker T)^{\perp}=\{0\}$. Therefore $\widetilde{T}$ is injective. The proof is complete.
	\end{proof}
	
	Since every $n$-power normal operator is $n$-power quasinormal, the following result shows that $T$ is $n$-power normal if and only if it is $n$-power quasinormal and $X=P_{\overline{\operatorname{ran}T^n}}T|_{\ker(T^\ast)^n}=0$ in \eqref{mr5}. An operator $T$ is called a \textit{quasiaffinity} if it is injective and has dense range.
	
	\begin{theorem}\label{coromain}
		Let $n\in\mathbb N$, and let $T\in\mathbb{B}(\mathcal H)$ be an
		$n$-power quasinormal operator. Then the following conditions are
		equivalent:
		\begin{itemize}
			\item[(i)]
			$\overline{\operatorname{ran}T^n}=\overline{\operatorname{ran}T^{n+1}}$.
			
			\item[(ii)]
			$\ker (T^\ast)^n\subseteq\ker T^n$.
			
			\item[(iii)]
			$P_{\overline{\operatorname{ran}T^n}}
			T\big|_{\ker (T^\ast)^n}=0$.
			
			\item[(iv)]
			$T=\widetilde T\oplus N$, where $\widetilde T$ and $N$ are as in \eqref{mr5}; moreover, $\widetilde T$ is an $n$-power normal quasiaffinity.
			
			\item[(v)]
			$T^n$ is normal.
		\end{itemize}
	\end{theorem}
	\begin{proof}
		We shall prove the implications in the order
		$${\rm (i)}\Longrightarrow{\rm (ii)}
		\Longrightarrow{\rm (iii)}
		\Longrightarrow{\rm (iv)}
		\Longrightarrow{\rm (v)}
		\Longrightarrow{\rm (i)}.$$	
		(i) $\Rightarrow$ (ii). From \eqref{mr5} we have
		\begin{equation}\label{coro2.3eq1}
			\overline{\operatorname{ran}T^{n+1}}=\overline{T(\operatorname{ran}T^n)}=\overline{\widetilde T(\operatorname{ran}T^n)}=\overline{\operatorname{ran}\widetilde T}.
		\end{equation}
		The assumption then yields $\overline{\operatorname{ran}\widetilde T}=\overline{\operatorname{ran}T^n}$, so $\ker\widetilde T^*=\{0\}$. Since $\widetilde T^*X=0$, we obtain $X=0$, hence $T=\widetilde T\oplus N$. For any $x\in\ker (T^\ast)^n$, we have $T^nx=N^nx=0$ because $N^n=0$. Therefore $\ker(T^\ast)^n\subseteq\ker T^n$.
		
		(ii) $\Rightarrow$ (iii). From \eqref{mr5} and the assumption, we have $T^n={\widetilde{T}}^n\oplus0$, which yields 
		\begin{equation}\label{coro1.2}
			\overline{\operatorname{ran}T^n}=\overline{\operatorname{ran}{\widetilde T}^{n}}.
		\end{equation}
		Together with the inclusions $\operatorname{ran}{\widetilde{T}}^n\subseteq\operatorname{ran}\widetilde{T}\subseteq\overline{\operatorname{ran}T^n}$, we obtain
		\begin{equation}\label{coro1.3}
			\overline{\operatorname{ran}\widetilde{T}}=\overline{\operatorname{ran}T^n}.
		\end{equation}
		On the other hand, from $T^nT^\ast T=T^\ast TT^n$, a direct computation gives $\begin{bmatrix}
			\times&\times\\0&\times
		\end{bmatrix}=\begin{bmatrix}
			\times&\times\\X^\ast{\widetilde{T}}^{n+1}&\times
		\end{bmatrix}$, so that
		\begin{equation}\label{coro1.4}
			X^\ast{\widetilde{T}}^{n+1}=0.
		\end{equation}
		Moreover, from \eqref{coro1.2},  $$\operatorname{ran}\widetilde{T}=\widetilde{T}(\overline{\operatorname{ran}T^n})=\widetilde{T}(\overline{\operatorname{ran}{\widetilde{T}}^n})\subseteq\overline{\widetilde T(\operatorname{ran}\widetilde T^{n})}=\overline{\operatorname{ran}{\widetilde{T}}^{n+1}}\subseteq\overline{\operatorname{ran}\widetilde{T}}.$$ Consequently, \eqref{coro1.3} yields $\overline{\operatorname{ran}{\widetilde{T}}^{n+1}}=\overline{\operatorname{ran}T^n}$. Applying \eqref{coro1.4} then gives $X^\ast=0$ and thus $X=0$. Therefore $P_{\overline{\operatorname{ran}T^n}}
		T\big|_{\ker(T^\ast)^n}=0$.
		
		(iii) $\Rightarrow$ (iv).
		From $X=P_{\overline{\operatorname{ran}T^n}}
		T|_{\ker (T^\ast )^n}$ and \eqref{mr5} we have $T=\widetilde{T}\oplus N$. Hence $T^n = \widetilde{T}^n \oplus 0$, and consequently as in the proof of (ii) $\Rightarrow$ (iii) we obtain \eqref{coro1.3}; that is, $\widetilde{T}$ has dense range. Since $T$ is $n$-power quasinormal, we have
		$$\widetilde T^n\widetilde T^\ast\widetilde{T}=\widetilde T^\ast\widetilde{T}\widetilde T^n=\widetilde T^\ast \widetilde T^n\widetilde{T}.$$
		The denseness of the range of $\widetilde{T}$ allows right cancellation of $\widetilde{T}$, which yields
		$$\widetilde{T}^n\widetilde T^\ast=\widetilde T^\ast \widetilde{T}^n.$$
		Applying the above equality repeatedly gives $\widetilde{T}^n(\widetilde T^\ast)^n=(\widetilde T^\ast)^n \widetilde{T}^n$, i.e., $\widetilde{T}^n$ is normal, as desired.	
		
		(iv) $\Rightarrow$ (v).
		The desired result follows immediately from $T^n=\widetilde{T}^n\oplus0$.
		
		(v) $\Rightarrow$ (i).
		By the assumption, \begin{equation}\label{coro1.5}
			\ker T^n=\ker(T^\ast)^n.
		\end{equation}
		Since $T$ commutes with $T^n$, the subspace $\ker T^n$ is invariant under $T$. Therefore, from \eqref{coro1.5} it follows that $X=P_{\overline{\operatorname{ran}T^n}}
		T\big|_{\ker (T^\ast)^n}=0$ in
		\eqref{mr5}.
		Hence $T=\widetilde T\oplus N$, and consequently $T^n=\widetilde T^n\oplus0$. As in the proof of (ii) $\Rightarrow$ (iii), we obtain \eqref{coro1.3}; that is, $\widetilde{T}$ has dense range. Then, by
		\eqref{coro2.3eq1} we have $\overline{\operatorname{ran}T^{n+1}}
		=\overline{\operatorname{ran}\widetilde T}
		=\overline{\operatorname{ran}T^n}$. This proves (i) and completes the proof.
	\end{proof}
	\begin{remark}
		(i) The above proof shows that the implication $$T^n\ \text{is normal}\Rightarrow P_{\overline{\operatorname{ran}T^n}}
		T|_{\ker(T^\ast)^n}=0 $$does not require $T$ to be $n$-power quasinormal.
		
		(ii) In general, the closures in condition (i) cannot be omitted. Consider the diagonal operator $T$ on $l^2(\mathbb{N}_{0})$ defined by $Te_{k}=\frac{1}{k+1}e_{k}$ for $k\in\mathbb{N}_{0}$. Then $T$ is self-adjoint, hence $n$-power quasinormal, and $T^n$ is normal for every $n\in\mathbb{N}$. However, $y\in\operatorname{ran}T^n$ but $y\notin\operatorname{ran}T^{n+1}$, where $y=\sum\limits_{k=0}^\infty \frac{1}{(k+1)^{n+1}}e_{k}$. Indeed, one readily verifies that $T^n\Big(\sum\limits_{k=0}^\infty\frac{1}{k+1}e_{k}\Big)=y$. If $y\in\operatorname{ran}T^{n+1}$, then for some $x=\sum\limits_{k=0}^\infty x_{k}e_{k}\in l^2({\mathbb{N}_{0}})$, we have $T^{n+1}x=y$, which implies $x_{k}=1$ for all $k$. Thus $x=\sum\limits_{k=0}^\infty e_{k}\not\in l^2(\mathbb{N}_{0})$, contradiction. Hence $y\notin\operatorname{ran}T^{n+1}$.
		
		(iii) We give an example to show that, without the $n$-power quasinormality
		assumption, conditions (i), (ii) and (iii) in the above theorem do not imply that $T^n$ is normal. Let $T=
		\begin{bmatrix}
			1&1\\
			0&1
		\end{bmatrix}
		\in \mathbb{M}_{2}(\mathbb C)$. For every $n\geq1$, $T^n=
		\begin{bmatrix}
			1&n\\
			0&1
		\end{bmatrix}$. A direct computation gives
		$$T^n(T^*T)-(T^*T)T^n=
		\begin{bmatrix}
			n&n\\
			0&-n
		\end{bmatrix}
		\neq0.$$ Hence $T$ is not $n$-power quasinormal. Since $T$ is invertible, we have $\overline{\operatorname{ran}T^n}
		=\mathbb C^2$ and $\ker(T^*)^n=\{0\}$. Therefore, $P_{\overline{\operatorname{ran}T^n}}T|_{\ker(T^*)^n}=0$, $\ker (T^\ast)^n\subseteq\ker T^n$, and also $\overline{\operatorname{ran}T^n}=\overline{\operatorname{ran}T^{n+1}}=\mathbb C^2$. Moreover, a straightforward induction gives $(T^\ast)^nT^n=\begin{bmatrix}
			1&n\\
			n&n^2+1
		\end{bmatrix}$ and $T^n(T^\ast)^n=
		\begin{bmatrix}
			n^2+1&n\\
			n&1
		\end{bmatrix}$. Thus, $(T^\ast)^nT^n\neq T^n(T^\ast)^n$, so $T^n$ is not normal.	
	\end{remark}
	
	Recall that the {\em descent} $\operatorname{desc}(T)$ of an operator $T$ is the least $n\in\mathbb{N}_{0}$ such that $\operatorname{ran}T^n=\operatorname{ran}T^{n+1}$ (with $\operatorname{desc}(T)=\infty$ if no such $n$ exists), and that the {\em nilpotency index} $\operatorname{ind}(N)$ of a nilpotent operator $N$ is the least $m\in\mathbb{N}$ with $N^m=0$. In view of Proposition \ref{MR}, it is natural to consider the special case of \eqref{mr5} where $\widetilde{T}$ is invertible and $\operatorname{ind}(N)=n$. In fact, we have the following.
	\begin{corollary}\label{desc}
		Let $n\in\mathbb N$, and let $T\in\mathbb{B}(\mathcal H)$ be an
		$n$-power quasinormal operator. Then $\operatorname{desc}(T)=n$ if and only if $\widetilde{T}$ is invertible and $\operatorname{ind}(N)=n$, where $\widetilde T$ and $N$ are as in \eqref{mr5}. In this case, $T^n$ is normal.
	\end{corollary}
	\begin{proof}
		{\em Necessity}. Suppose that $\operatorname{desc}(T)=n$. Then
		\begin{equation}\label{desequ}
			\operatorname{ran}T^{n-1}\ne\operatorname{ran}T^n
		\end{equation} and
		$\overline{\operatorname{ran}T^n}=\overline{\operatorname{ran}T^{n+1}}$. By Theorem \ref{coromain}, we have $T=\widetilde T\oplus N$, where $\widetilde T$ is an $n$-power normal quasiaffinity and $N^n=0$. Consequently, $\operatorname{ran}\widetilde{T}^n=\operatorname{ran}T^n=\operatorname{ran}T^{n+1}=\operatorname{ran}\widetilde{T}^{n+1}$. Hence, $$\widetilde{T}^n(\overline{\operatorname{ran}T^n})=\widetilde{T}^{n+1}(\overline{\operatorname{ran}T^n})=\widetilde{T}^n(\widetilde{T}\overline{\operatorname{ran}T^n}).$$
		Since $\widetilde T$ is a quasiaffinity, it is injective, and so is $\widetilde T^{n}$. Therefore, by the above equality, we obtain $\overline{\operatorname{ran}T^n}=\widetilde{T}\overline{\operatorname{ran}T^n}$, so $\widetilde T$ is surjective, and hence invertible. This implies 
		$$\operatorname{ran}T^{n-1}=\operatorname{ran}\widetilde{T}^{n-1}\oplus\operatorname{ran}N^{n-1}=\operatorname{ran}\widetilde{T}\oplus\operatorname{ran}N^{n-1}$$
		and $$\operatorname{ran}T^n=\operatorname{ran}\widetilde{T}^n=\operatorname{ran}\widetilde{T}.$$ By \eqref{desequ}, we have $\operatorname{ran}N^{n-1}\ne\{0\}$, so $N^{n-1}\ne0$. Together with $N^n=0$, this yields $\operatorname{ind}(N)=n$.
		
		{\em Sufficiency}. Since $\widetilde{T}$ is invertible, it follows from $\widetilde{T}^\ast X=0$ in Proposition \ref{MR} that $X=0$; hence $T=\widetilde{T}\oplus N$. Therefore, $\operatorname{ran}T^n=\operatorname{ran}\widetilde{T}^n=\operatorname{ran}\widetilde{T}^{n+1}=\operatorname{ran}T^{n+1}$. Thus $\operatorname{desc}(T)\le n$. On the other hand, $$\operatorname{ran}T^{n-1}=\operatorname{ran}\widetilde{T}^{n-1}\oplus\operatorname{ran} N^{n-1}=\operatorname{ran}\widetilde{T}^n\oplus\operatorname{ran}N^{n-1}.$$ Since $\operatorname{ind}(N)=n$, we have $\operatorname{ran}N^{n-1}\neq\{0\}$, and hence $\operatorname{ran}T^{n-1}
		\neq
		\operatorname{ran}\widetilde{T}^n=\operatorname{ran}T^n$.
		Therefore $\operatorname{desc}(T)\ge n$, and consequently $\operatorname{desc}(T)=n$.
		
		In this case, Theorem \ref{coromain} (i) immediately gives that $T^n$ is normal.
	\end{proof}
	
	 The unilateral shift $W=\operatorname{shift}(1,1,1,\cdots)$ is an isometry ($W^\ast W=I$), hence $n$-power quasinormal. Its spectrum $\sigma(W)$ is the closed unit disk, while its approximate point spectrum $\sigma_{ap}(W)$ is the unit circle. Ko and Lee proved the latter part of the result below in \cite[Corollary 3.14]{ko-filomat-2023} via the single-valued extension property of $n$-power normal operators; here we give an alternative proof and establish a stronger result.
	
	\begin{corollary}
		Let $n\in\mathbb N$, and let $T\in\mathbb{B}(\mathcal H)$ be an
		$n$-power quasinormal operator. If $\operatorname{ran}T=\operatorname{ran}T^{n+1}$, then $T$ is similar to a normal operator, so $\sigma(T)=\sigma_{ap}(T)$.
	\end{corollary}
	\begin{proof} 
		Since $\operatorname{ran}T^{n+1} \subseteq \operatorname{ran}T^n \subseteq \cdots \subseteq \operatorname{ran}T$ and \(\operatorname{ran}T = \operatorname{ran}T^{n+1}\), it follows that
		\begin{equation}\label{rangepower}
			\operatorname{ran}T = \operatorname{ran}T^2 = \cdots = \operatorname{ran}T^{n+1}.
		\end{equation}Hence, \(\overline{\operatorname{ran}T^n} = \overline{\operatorname{ran}T^{n+1}}\). Therefore, Theorem \ref{coromain} yields \(T = \widetilde{T} \oplus N\), where $\widetilde T$ is an $n$-power normal quasiaffinity and $N^n=0$. By \eqref{rangepower}, we have
		$\operatorname{ran}\widetilde{T}^n=\operatorname{ran}T^n=\operatorname{ran}T^{n+1}=\operatorname{ran}\widetilde{T}^{n+1}$. Following the same argument as in the proof of necessity in Corollary \ref{desc}, we get that $\widetilde{T}$ is surjective; hence $\widetilde{T}^n$ is an invertible normal operator. Applying \cite[Theorem 1]{kur-mz-1962}, we obtain that $\widetilde{T}$ is similar to a normal operator. On the other hand, from \eqref{rangepower} we have $\operatorname{ran}T^n=\operatorname{ran}T=\operatorname{ran}\widetilde{T}\oplus\operatorname{ran}N$. Together with $\operatorname{ran}N\subseteq\ker (T^\ast)^n$, this yields $N=0$. Consequently, $T=\widetilde{T}\oplus0$ is similar to a normal operator $R$, which implies $\sigma(T)=\sigma(R)$ and $\sigma_{ap}(T)=\sigma_{ap}(R)$. The normality of $R$ gives $\sigma(R)=\sigma_{ap}(R)$, hence $\sigma(T)=\sigma_{ap}(T)$. The proof is complete.
	\end{proof}
	
	\begin{corollary}{\cite[Lemma 2.4]{stankovic-afa-2025}}\label{quaisnormal}
		Let $T\in\mathbb{B}(\mathcal{H})$. Then $T$ is normal if and only if $T$ is quasinormal and $\ker T^\ast\subseteq\ker T$.
	\end{corollary}
	\begin{proof}
		Since $1$-power quasinormal operators are quasinormal, the desired result follows from Theorem \ref{coromain} (ii) and (v).	
	\end{proof}
	We now give the correct formulation of Theorem \ref{kolee} in the theorem below. We then conclude this section with a new proof of Theorem \ref{svep} \cite[Theorem 3.2]{ko-filomat-2023}, which avoids the use of \cite[Lemma 3.1]{ko-filomat-2023} and thereby removes the gap in the original argument.	
		\begin{theorem}\label{mainthm}
		Let $n\in\mathbb{N}$ and $T\in\mathbb{B}(\mathcal{H})$. Consider the following conditions:
		\begin{itemize}
			\item[(i)] $T^n$ is normal.
			\item[(ii)] $T^n$ is quasinormal.
			\item[(iii)] $T$ is $n$-power quasinormal.
		\end{itemize}
		Then $\mathrm{(i)}$ implies $\mathrm{(ii)}$ and $\mathrm{(iii)}$. Morever, if $$\ker (T^\ast)^n\subseteq\ker T^n,$$ then the three conditions are equivalent. In particular, if $\mathcal{H}$ is finite-dimensional, the above kernel inclusion is automatically satisfied whenever any one of $\mathrm{(i)}$--$\mathrm{(iii)}$ holds.
	\end{theorem} 
	\begin{proof}
		By Corollary \ref{quaisnormal} and the fact that every $n$--power normal operator is $n$-power quasinormal, it follows that (i) yields both (ii) and (iii). Under the condition $\ker (T^\ast)^n\subseteq\ker T^n$, Corollary \ref{quaisnormal} and Theorem \ref{coromain} show, respectively, that (ii) and (iii) each imply the normality of $T^n$. Moreover, as is well known, quasinormality and normality are equivalent in finite-dimensional spaces; hence, in this case, the kernel inclusion follows immediately from either (i) or (ii). Finally, let $T$ be as in \eqref{mr5}. Since $\widetilde{T}$	is injective on the finite-dimensional space $\operatorname{ran}T^n$, it is invertible. In view of $\widetilde{T}^\ast X=0$, this implies that $X=P_{\overline{\operatorname{ran}T^n}}T|_{\ker (T^\ast)^n}=0$. Theorem \ref{coromain} then immediately gives $\ker (T^\ast)^n\subseteq\ker T^n$. The proof is complete.
	\end{proof}
	
	To prove the next theorem, we need the following proposition.
	
	\begin{proposition}\label{svep1}
		Let $n\in\mathbb{N}$, and let $T$ be an $n$-power quasinormal operator. For distinct nonzero complex numbers $\alpha$ and $\beta$, we have $\ker (T^n-\alpha I)\perp\ker (T^n-\beta I)$.
	\end{proposition}
	\begin{proof}
		If we can prove that $\mathcal{S}_{z}\triangleq\ker(T^n-zI)$ is invariant under $T^\ast$ for every nonzero $z\in\mathbb{C}$, then since $T^n=zI$ on $\mathcal{S}_{z}$, it follows that $(T^n)^\ast=\bar{z}I$ on $\mathcal{S}_{z}$. Indeed, from $T^nT^\ast T=T^\ast TT^n$, we have
		\begin{equation}\label{inv1}
			T^\ast T\mathcal{S}_{z}\subseteq\mathcal{S}_{z}.
		\end{equation} 
		Since $T^n=zI$ on $\mathcal{S}_{z}$, for every $x\in\mathcal{S}_{z}$, there exists $y=\frac{1}{z}T^{n-1}x\in\mathcal{S}_{z}$ such that $Ty=x$. Then \eqref{inv1} gives $T^\ast x=T^\ast Ty\in\mathcal{S}_{z}$. Hence, $T^\ast\mathcal{S}_{z}\subseteq\mathcal{S}_{z}$, as desired. Therefore, for any $x\in\mathcal{S}_{\alpha}$ and $y\in \mathcal{S}_{\beta}$, we get
		$$\alpha\langle x,y\rangle=\langle T^nx,y\rangle=\langle x,(T^n)^\ast y\rangle=\langle x,\bar{\beta}y\rangle=\beta\langle x,y\rangle.$$
		Since $\alpha\ne\beta$, it follows that $\langle x,y\rangle=0$. Hence, $\mathcal{S}_{\alpha}\perp\mathcal{S}_{\beta}$.
	\end{proof}
	
	An operator $T\in\mathbb{B}(\mathcal{H})$ is said to have the {\em single-valued extension property} if, for every open set $U\subseteq\mathbb{C}$, the only analytic function $f:U\to\mathcal{H}$ satisfying $(T-z)f(z)=0$ for all $z\in U$ is the zero function. 
	
	\begin{theorem}\label{svep}
		Let $n\in\mathbb{N}$, and let $T\in\mathbb{B}(\mathcal{H})$ be an $n$-power quasinormal operator. Then $T$ has the single-valued extension property.
	\end{theorem}
	\begin{proof}
		Let $U\subseteq\mathbb{C}$ be open, and suppose $f:U\to\mathcal{H}$ is analytic and satisfies $(T-zI)f(z)=0$ for all $z\in U$. Then $Tf(z)=zf(z)$, and hence $T^nf(z)=z^nf(z)$. Therefore, $f(z)\in\ker(T^n-z^n I)$. For any nonzero $z_{0}\in U$, we can choose a sequence $\{z_{k}\}_{k\in\mathbb{N}}$ of nonzero complex numbers in $U$ such that $\lim\limits_{k\to\infty}z_{k}=z_{0}$ and $z_{k}^n\ne z_{0}^n$ for all $k$. By Proposition \ref{svep1}, we have $f(z_{k})\perp f(z_{0})$, i.e., $\langle f(z_{k}), f(z_{0})\rangle=0$. Since $f$ is analytic, it is continuous. Hence, $$\|f(z_{0})\|^2=\lim\limits_{k\to\infty}\langle f(z_{k}),f(z_{0})\rangle=0\Rightarrow f(z_{0})=0.$$By the arbitrariness of $z_{0}\in U\backslash\{0\}$ and the continuity of $f$, it follows that $f(z)=0$ for all $z\in U$. The proof is complete.
	\end{proof}
	
		\section{Power-quasinormality index sets}
		
In this section, we use the operator matrix representation from Proposition \ref{MR} to determine the positive integers $n$ for which $T$ is $n$-power quasinormal. For this purpose, we define the {\em power-quasinormality index set} of $T\in\mathbb{B}(\mathcal{H})$ as 
	$$\mathcal{PQ}(T)=\{n\in\mathbb{N}:[T^n,T^\ast T]=0\}.$$
	 It is not surprising that $\mathcal{PQ}(T)$ can be empty (see Example \ref{pqsetempty}), since not every operator is $n$-power quasinormal for some $n\in\mathbb{N}$. We begin with the basic properties of $\mathcal{PQ}(T)$.
	
	\begin{proposition}\label{baspqset}
		Let $T\in\mathbb{B}(\mathcal{H})$. Then the following properties hold:
		\begin{itemize}
			\item[(i)] If $\lambda$ is a nonzero complex number, then $\mathcal{PQ}(\lambda T)=\mathcal{PQ}(T)$.
			\item[(ii)] If $U\in\mathbb{B}(\mathcal{H})$ is unitary, then $\mathcal{PQ}(UTU^\ast)=\mathcal{PQ}(T)$.
			\item[(iii)] If $S\in\mathbb{B}(\mathcal{H})$, then $\mathcal{PQ}(T\oplus S)=\mathcal{PQ}(T)\bigcap\mathcal{PQ}(S).$
			\item[(iv)] If $m,n\in\mathcal{PQ}(T)$, then $m+n\in\mathcal{PQ}(T)$.
			\item[(v)] If $T$ is injective and $m,n\in\mathcal{PQ}(T)$ with $m-n\in\mathbb{N}$, then $m-n\in\mathcal{PQ}(T)$.
		\end{itemize}
	\end{proposition}
	\begin{proof}
		Only (v) requires proof; the others follow directly from the definition of $\mathcal{PQ}(T)$. From $m,n\in\mathcal{PQ}(T)$, it follows that 
		$$T^n(T^{m-n}T^\ast T-T^\ast TT^{m-n})=T^mT^\ast T-T^nT^\ast TT^{m-n}=T^mT^\ast T-T^\ast TT^{m}=0.$$
		Since $T$ is injective, so is $T^n$. Thus, from the above equation, $T^{m-n}T^\ast T-T^\ast TT^{m-n}=0$, as desired.
	\end{proof}
	\begin{remark}\label{squarezero}
		The injectivity of $T$ in (v) cannot be removed. Indeed, for the non-invertible matrix $T=\begin{bmatrix}
			0&1\\0&0
		\end{bmatrix}$, we have $T^2=0$, so $\mathcal{PQ}(T)=\{2,3,4,\cdots\}$, but $1=3-2\not\in\mathcal{PQ}(T)$.
	\end{remark}
	We now provide an explicit description of $\mathcal{PQ}(T)$.
	\begin{theorem}\label{mr4}
		Let $n\in\mathbb{N}$, and let $T\in\mathbb{B}(\mathcal{H})$ be an $n$-power quasinormal operator. Then there exist unique positive integers $d$ and $q$ such that 
		$$\mathcal{PQ}(T)=\{d(q+k):k\in\mathbb{N}_{0}\}=\{dq,d(q+1),d(q+2),\cdots\}.$$
		More precisely, $d=\min\mathcal{PQ}(\widetilde{T})$ and $q=\frac{m}{d}$, where $m=\min\mathcal{PQ}(T)$ and $\widetilde{T}$ is given by \eqref{pqset}.	
	\end{theorem}	
	\begin{proof}
		Since $T$ is $n$-power quasinormal, we have $n\in\mathcal{PQ}(T)$; hence $\mathcal{PQ}(T)$ is nonempty. Let $m=\min\mathcal{PQ}(T)$. Then $T$ is $m$-power quasinormal, and by Proposition \ref{MR}, \begin{equation}\label{pqset}
			T=\begin{bmatrix}
				\widetilde{T}&X\\0&N
			\end{bmatrix}
		\end{equation}
		with respect to the space decomposition $\mathcal{H}=\overline{\operatorname{ran}T^m}\oplus\ker (T^\ast)^m$. Since $\widetilde{T}$ is an injective $m$-power quasinormal operator, we have $m\in\mathcal{PQ}(\widetilde{T})$, so $\mathcal{PQ}(\widetilde{T})$ is nonempty. Let $d=\min\mathcal{PQ}(\widetilde{T})$. If $p\in\mathcal{PQ}(\widetilde{T})$, then by the division algorithm, $p=ad+r$ for some integers $a$ and $r$ with $0\le r<d$. Proposition \ref{baspqset} (v) implies $r\in\mathcal{PQ}(\widetilde{T})$, since $r=p-ad$ and $p,d\in\mathcal{PQ}(\widetilde{T})$. Since $d$ is the minimum of $\mathcal{PQ}(\widetilde{T})$, we must have $r=0$; otherwise, a contradiction. Thus $p=ad$, so by Proposition \ref{baspqset} (iv), $\mathcal{PQ}(\widetilde{T})=\{d,2d,3d,\cdots\}$. In particular, since $m\in\mathcal{PQ}(\widetilde{T})$, we have $m=dq$ for some $q\in\mathbb{N}$. Therefore, if we can prove $\mathcal{PQ}(T)=\mathcal{PQ}(\widetilde{T})\bigcap\{n\in\mathbb{N}:n\ge m\}$, then the desired result follows. Indeed, from $\widetilde{T}^\ast X=0$, together with \eqref{mr} and \eqref{mr2}, a direct computation shows that for every $k\in\mathbb{N}$, $k\in\mathcal{PQ}(T)$ if and only if $T^k T^\ast T=T^\ast TT^k$ if and only if
		\begin{equation}\label{pqset1}
			\begin{cases}
				\widetilde{T}^k\widetilde{T}^\ast\widetilde{T}=\widetilde{T}^\ast\widetilde{T}\widetilde{T}^k\\
				\widetilde{T}^\ast\widetilde{T}\sum\limits_{i=0}^{k-1}\widetilde{T}^{k-1-i}XN^{i}=\sum\limits_{i=0}^{k-1}\widetilde{T}^{k-1-i}XN^{i}(X^\ast X+N^\ast N)\\
				N^k(X^\ast X+N^\ast N)=(X^\ast X+N^\ast N)N^k
			\end{cases}.
		\end{equation}Consequently, if $k_{1}\in\mathcal{PQ}(T)$, then the first equation above together with  $m=\min\mathcal{PQ}(T)$ yields $k_{1}\in\mathcal{PQ}(\widetilde{T})\bigcap\{n\in\mathbb{N}:n\ge m\}$. Thus $\mathcal{PQ}(T)\subseteq\mathcal{PQ}(\widetilde{T})\bigcap\{n\in\mathbb{N}:n\ge m\}$. For the reverse inclusion, take $k_{2}\in\mathcal{PQ}(\widetilde{T})\bigcap\{n\in\mathbb{N}:n\ge m\}$ (If $k_{2}=m$, then clearly $k_{2}\in\mathcal{PQ}(T)$; hence we assume $k_{2}>m$). Then
		\begin{equation}\label{pqset2}
			\widetilde{T}^{k_2}\widetilde{T}^\ast\widetilde{T}=\widetilde{T}^\ast\widetilde{T}\widetilde{T}^{k_2}.
		\end{equation} Since $N^m=0$ and $k_{2}> m$, it follows that 
		\begin{equation}\label{pqset3}
			N^{k_{2}}(X^\ast X+N^\ast N)=(X^\ast X+N^\ast N) N^{k_{2}}.
		\end{equation} Moreover, because $m\in\mathcal{PQ}(T)$, the second equation in \eqref{pqset1} gives \begin{equation}\label{commute3}
			\widetilde{T}^\ast\widetilde{T}\sum\limits_{i=0}^{m-1}\widetilde{T}^{m-1-i}XN^{i}=\sum\limits_{i=0}^{m-1}\widetilde{T}^{m-1-i}XN^{i}(X^\ast X+N^\ast N).
		\end{equation} 
 Since $m\in\mathcal{PQ}(\widetilde{T})$, Proposition \ref{baspqset} (v) yields  $k_{2}-m\in\mathcal{PQ}(\widetilde{T})$, i.e.,
		\begin{equation}\label{commute1}
			\widetilde{T}^{k_{2}-m}\widetilde{T}^\ast\widetilde{T}=\widetilde{T}^\ast\widetilde{T}\widetilde{T}^{k_{2}-m}.
		\end{equation}
		Also, from $N^m=0$, we have \begin{equation}\label{commute2}
			\sum\limits_{i=0}^{k_{2}-1}\widetilde{T}^{k_{2}-1-i}XN^{i}=\sum\limits_{i=0}^{m-1}\widetilde{T}^{k_{2}-1-i}XN^{i}+\sum\limits_{i=m}^{k_{2}-1}\widetilde{T}^{k_{2}-1-i}XN^{i}=\widetilde{T}^{k_{2}-m}\sum\limits_{i=0}^{m-1}\widetilde{T}^{m-1-i}XN^{i}.
		\end{equation}
		Using \eqref{commute3}, \eqref{commute1} and \eqref{commute2}, we obtain
		\begin{align*}
			\widetilde{T}^\ast\widetilde{T}\sum\limits_{i=0}^{k_{2}-1}\widetilde{T}^{k_{2}-1-i}XN^{i}&=\widetilde{T}^\ast\widetilde{T}\widetilde{T}^{k_{2}-m}\sum\limits_{i=0}^{m-1}\widetilde{T}^{m-1-i}XN^{i}=\widetilde{T}^{k_{2}-m}\widetilde{T}^\ast\widetilde{T}\sum\limits_{i=0}^{m-1}\widetilde{T}^{m-1-i}XN^{i}\\&=\widetilde{T}^{k_{2}-m}\sum\limits_{i=0}^{m-1}\widetilde{T}^{m-1-i}XN^{i}(X^\ast X+N^\ast N)\\&=\sum\limits_{i=0}^{m-1}\widetilde{T}^{k_{2}-1-i}XN^{i}(X^\ast X+N^\ast N).
		\end{align*}
		Together with $N^m=0$, this gives
		$$\widetilde{T}^\ast\widetilde{T}\sum\limits_{i=0}^{k_{2}-1}\widetilde{T}^{k_{2}-1-i}XN^{i}=\sum\limits_{i=0}^{k_{2}-1}\widetilde{T}^{k_{2}-1-i}XN^{i}(X^\ast X+N^\ast N).$$
		Combining this identity with \eqref{pqset2}, \eqref{pqset3} and \eqref{pqset1}, we conclude $k_{2}\in\mathcal{PQ}(T)$. This proves the reverse inclusion, and hence the desired equality follows.
	\end{proof}	
	
	A direct consequence of Theorem \ref{mr4} is a characterization of the set of all $n\in\mathbb{N}$ for which $T^n$ is normal.
	\begin{corollary}
	Let $n\in\mathbb{N}$, and let $T\in\mathbb{B}(\mathcal{H})$ be an $n$-power normal operator. Then there exist unique positive integers $d$ and $q$ such that $$\{k\in\mathbb{N}:[T^k,T^\ast]=0\}=\{d(q+k_{0}),d(q+k_{0}+1),d(q+k_{0}+2),\cdots\},$$ where $k_{0}=\min\{k\in\mathbb{N}_{0}:\overline{\operatorname{ran}T^{d(q+k)}}=\overline{\operatorname{ran}T^{d(q+k)+1}}\}$.
	\end{corollary}
	\begin{proof}
Since every $n$-power normal operator is $n$-power quasinormal, by Theorem \ref{mr4} there exist unique positive integers $d$ and $q$ such that$$\{k\in\mathbb{N}:[T^k,T^\ast]=0\}\subseteq\{dq,d(q+1),d(q+2),\cdots\}.$$
By Theorem \ref{coromain}, we have $k_{0}=\min\{k\in\mathbb{N}_{0}:\overline{\operatorname{ran}T^{d(q+k)}}=\overline{\operatorname{ran}T^{d(q+k)+1}}\}$. Hence,$$\{k\in\mathbb{N}:[T^k,T^\ast]=0\}\subseteq\{d(q+k_{0}),d(q+k_{0}+1),d(q+k_{0}+2),\cdots\}.$$ 
Therefore, by Theorem \ref{coromain}, it suffices to prove that $\overline{\operatorname{ran}T^{d(q+k_{0}+j)}}=\overline{\operatorname{ran}T^{d(q+k_{0}+j)+1}}$ for $j\in\mathbb{N}$. Indeed, from $\overline{\operatorname{ran}T^{d(q+k_{0})}}=\overline{\operatorname{ran}T^{d(q+k_{0})+1}}$ we readily obtain $$\overline{\operatorname{ran}T^{d(q+k_{0})}}=\overline{\operatorname{ran}T^{d(q+k_{0})+1}}=\overline{\operatorname{ran}T^{d(q+k_{0})+2}}=\cdots=\overline{\operatorname{ran}T^{d(q+k_{0}+j)}}=\overline{\operatorname{ran}T^{d(q+k_{0}+j)+1}}=\cdots$$by induction, as desired.
	\end{proof}
	
The next two corollaries extend \cite[Proposition 2.6]{ahmed-bmaa-2011}.	
	
	\begin{corollary}
		Let $T\in\mathbb{B}(\mathcal{H})$ and $s,t\in \mathcal{PQ}(T)$ with $s>t$. Then $T$ is $t+k(s-t)$-power quasinormal for $k\in\mathbb{N}_{0}$. 
	\end{corollary}
\begin{proof}
From Theorem \ref{mr4}, we obtain positive integers $d$, $q$ and non-negative integers $k_{1}$, $k_{2}$ such that $d(q+k_{1})=s$ and $d(q+k_{2})=t$. Then, for every $k\in\mathbb{N}_{0}$, choosing $k_{3}=k_{2}+k(k_{1}-k_{2})$ gives$$d(q+k_{3})=t+k(s-t)\in\mathcal{PQ}(T),$$
as desired.
\end{proof}

	\begin{corollary}
		Let $T\in\mathbb{B}(\mathcal{H})$. If $s,t\in \mathcal{PQ}(T)$ are coprime with $s>t$, then $T$ is $n$-power quasinormal for every positive integer $n\ge t$. 
	\end{corollary}
\begin{proof}
From Theorem \ref{mr4}, there exist unique positive integers $d$ and $q$ such that $$\mathcal{PQ}(T)=\{dq,d(q+1),d(q+2),\cdots\}.$$ 
Since $s$ and $t$ are coprime, we have $d=1$. Consequently,$$\mathcal{PQ}(T)=\{q,q+1,q+2,\cdots,t-1,t,t+1,t+2,\cdots\},$$
as desired.
\end{proof}	
	The following generalizes \cite[Theorem 2.17]{ahmed-bmaa-2011}.
	\begin{corollary}
		Let $T\in\mathbb{B}(\mathcal{H})$ with $\ker T\subseteq\ker T^\ast$. If $s,t\in\mathcal{PQ}(T)$ are coprime, then $T$ is quasinormal.
	\end{corollary}
	\begin{proof}		
	Since $\ker T\subseteq\ker T^\ast$, we readily obtain $P_{\ker T}T=0$. Consequently, with respect to the space decomposition $\mathcal{H}=\ker T\oplus \ker T^\perp$, $T$ has the matrix representation $T=\begin{pmatrix}
	0&0\\0&S
	\end{pmatrix}$. Moreover, since $s,t\in\mathcal{PQ}(T)$ are coprime, Theorem \ref{mr4} yields a positive integer $q$ such that $\mathcal{PQ}(T)=\{q,q+1,q+2,\cdots\}$. Hence $q,q+1\in\mathcal{PQ}(S)$. Since $S$ is injective, Proposition \ref{baspqset} (v) gives $1=q+1-q\in\mathcal{PQ}(S)$. Therefore $1\in\mathcal{PQ}(0\oplus S)=\mathcal{PQ}(T)$, i.e., $T$ is quasinormal.
	\end{proof}
	\begin{remark}
	If the condition $\ker T\subseteq\ker T^\ast$ is removed, the result generally fails. Indeed, in Example \ref{exdq}, taking $d=1$ and $q=2$, we readily see that the coprime numbers $2,3\in\mathcal{PQ}(T_{1,2})$, but $1\notin\mathcal{PQ}(T_{1,2})$. That is, $T_{1,2}$ is not quasinormal. 
	\end{remark}
	
	At the end of this section, we observe that, given any positive integers $d$ and $q$, one can find a matrix $T_{d,q}$ satisfying $$\mathcal{PQ}(T_{d,q})=\{dq,d(q+1),d(q+2),\cdots\}.$$
This allows us to readily obtain some facts. For example, there exists an operator $T_{2,1}$ which is $n$-power quasinormal precisely for even $n$, and fails to be $m$-power quasinormal for every odd $m$; there exists an operator $T_{3,1}$ which is $3$-power quasinormal but not $4$-power quasinormal, and so on.
	
	\begin{example}\label{exdq}
		Suppose $d,q\in\mathbb{N}$. Let $S_{d}=\begin{bmatrix}
			1&0\\0&1
		\end{bmatrix}$ if $d=1$, and $S_{d}=\begin{bmatrix}
			1&1\\0&\omega_{d}
		\end{bmatrix}$ if $d\ge2$, where $\omega_{d}=e^{\frac{2\pi\mathrm{i}}{d}}$. For $d\ge 2$ and $k\in\mathbb{N}$, a direct calculation gives $S_{d}^k=\begin{bmatrix}
			1&\frac{1-(\omega_{d})^k}{1-\omega_{d}}\\0&(\omega_{d})^k
		\end{bmatrix}$ and $S_{d}^\ast S_{d}=\begin{bmatrix}
			1&1\\1&2
		\end{bmatrix}$. Moreover, it is easy to verify that $[S_d^k, S_d^\ast S_d]=0$ if and only if $(\omega_d)^k=1$, which is equivalent to \(d \mid k\). The equality \(\mathcal{PQ}(S_1)=\{1,2,3,\dots\}\) is immediate. Consequently, for every $d\in\mathbb{N}$, $$\mathcal{PQ}(S_{d})=\{k\in\mathbb{N}:d\mid k\}=\{d,2d,3d,\cdots\}.$$ On the other hand, let $\{e_{i}\}_{i=\{1,2,\cdots,dq\}}$ be the standard basis of $\mathbb{C}^{dq}$, and let $J_{dq}$ be the nilpotent Jordan block of order $dq$. Then $J_{dq}e_{1}=0$ and $J_{dq}e_{i}=e_{i-1}$ for $i=2,3,\cdots,dq$. For $1\le k<dq$, one can readily get that $J_{dq}^k J_{dq}^* J_{dq} e_{k+1}=J_{dq}^k e_{k+1}=e_1$ but $J_{dq}^* J_{dq} J_{dq}^k e_{k+1} = 0$. Thus $k \notin\mathcal{PQ}(J_{dq})$. For $k \ge dq$, $J_{dq}^k = 0$, so $$\mathcal{PQ}(J_{dq})=\{k\in\mathbb{N}:k\ge dq\}=\{dq, dq+1, dq+2, \dots\}.$$
		Now let $T_{d,q} = S_d \oplus J_{dq}$. Proposition \ref{baspqset} (iii) then gives
		$$
		\mathcal{PQ}(T_{d,q})=\mathcal{PQ}(S_{d})\bigcap\mathcal{PQ}(J_{dq})=\{dq,d(q+1),d(q+2),\cdots\},
		$$
		as desired.	
	\end{example}
	
	\section{On the $n$-th root problem}
	
	Let $T=U|T|$ be the polar decomposition of $T$, and let $\lambda\in[0,1]$. Ch\={o} and Tanahashi \cite{Cho-Tanahashi-SMJ-2002} called the operator $\Delta_{\lambda}(T)=|T|^{\lambda}U|T|^{1-\lambda}$ the {\em $\lambda$-Aluthge transform} of $T$. In particular, $\Delta_{0}(T)=T$, $\Delta_{\frac{1}{2}}(T)$ is the  {\em Aluthge transform} of $T$ \cite{Aluthge-IEOT-1990} and $\Delta_{1}(T)=|T|U$ is the {\em Duggal transform} of $T$ \cite{Foias-PJM-2003}. The power quasinormality of an operator can be characterized in terms of the powers of its $\lambda$-Aluthge transform. 
	
	\begin{theorem}\label{char}
		Let $\lambda\in(0,1]$, $n\in\mathbb{N}$ and $T\in\mathbb{B}(\mathcal{H})$. Then $T$ is $n$-power quasinormal if and only if $(\Delta_{\lambda}(T))^n=T^n$.
	\end{theorem}
	\begin{proof}
		Let $T=U|T|$ be the polar decomposition of $T$. Since $\Delta_{\lambda}(T)|T|^\lambda=|T|^\lambda U|T|^{1-\lambda}|T|^\lambda=|T|^\lambda U|T|=|T|^\lambda T$, we obtain
		\begin{equation}\label{indentity}
			(\Delta_{\lambda}(T))^n|T|^\lambda=|T|^\lambda T^n
		\end{equation}
		for all $n\in\mathbb{N}$ by a simple induction.
		
		{\em Necessity}. Since $T$ is $n$-power quasinormal, we have $T^n|T|^2=|T|^2T^n$. By the continuous functional calculus, this implies $T^n|T|^\lambda=|T|^\lambda T^n$. Consequently, \eqref{indentity} yields $(\Delta_{\lambda}(T))^n|T|^\lambda=T^n|T|^\lambda$, that is, $$\Big((\Delta_{\lambda}(T))^n-T^n\Big)|T|^\lambda=0.$$
		Since $\overline{\operatorname{ran}|T|^\lambda}=(\ker |T|^\lambda)^\perp=(\ker |T|)^\perp$, the above identity gives  $(\Delta_{\lambda}(T))^n=T^n$ on $(\ker |T|)^\perp$. On the other hand, if $x\in\ker |T|$, then $T^nx=T^{n-1}U|T|x=0$. For $\lambda\in(0,1)$, we also have $(\Delta_{\lambda}(T))^nx=(\Delta_{\lambda}(T))^{n-1}|T|^\lambda U|T|^{1-\lambda}x=0$. Moreover, $(\Delta_{1}(T))^nx=(\Delta_{1}(T))^{n-1}|T|Ux=0$, since $\ker U=\ker |T|$. Hence $(\Delta_{\lambda}(T))^n=T^n$ holds on $\ker |T|$ as well. Therefore, we conclude that \((\Delta_{\lambda}(T))^n = T^n\) on $\mathcal{H}=\ker|T|\oplus(\ker|T|)^\perp$.
		
		{\em Sufficiency}. From the assumption we get $(\Delta_{\lambda}(T))^n|T|^\lambda=T^n|T|^\lambda$, and using \eqref{indentity} we obtain $$T^n|T|^\lambda=|T|^\lambda T^n.$$ By the continuous functional calculus, this implies $T^n|T|^2=|T|^2 T^n$, as desired.
	\end{proof}
	
	\begin{corollary}\cite[Proposition 3.5 (i)]{ko-filomat-2023}
	Let $n\in\mathbb{N}$, and let $\{T_{k}\}$ be a sequence of $n$-power quasinormal operators. If $T_{k}$ converges to $T$ in norm, then $T$ is also $n$-power quasinormal.
	\end{corollary}
	\begin{proof}
	Consider the map $f$ on $\mathbb{B}(\mathcal{H})$ defined by $f(T)=(\Delta_{\frac{1}{2}}(T))^n-T^n$. Since the map $T\mapsto\Delta_{\frac{1}{2}}(T)$ is continuous \cite[Corollary 2.6]{zhou-glma-2023}, $f$ is continuous. By Theorem \ref{char}, we have $f(T_{k})=0$ for each $k$. Since $\lim\limits_{k\to\infty} T_{k}=T$ in norm, it follows from the continuity of $f$ that$$f(T)=\lim\limits_{k\to\infty}f(T_{k})=0.$$
	Hence $(\Delta_{\frac{1}{2}}(T))^n=T^n$, which means that $T$ is $n$-power quasinormal.
	\end{proof}
	
 The following result describes those $n$-power quasinormal operators for which $T^n$ is quasinormal in terms of the compatibility of the $\lambda$-Aluthge transform with taking $n$-th powers.
	
	\begin{corollary}\label{erro}
		Let $\lambda\in(0,1]$, $n\in\mathbb{N}$ and $T\in\mathbb{B}(\mathcal{H})$. If $T$ is $n$-power quasinormal, then $T^n$ is quasinormal if and only if $(\Delta_{\lambda}(T))^n=\Delta_{\lambda}(T^n)$. 
	\end{corollary}
	\begin{proof}
		This result follows immediately from Theorem \ref{char} and the fact that quasinormal operators are fixed points of the map $T\mapsto\Delta_{\lambda}(T)$ on $\mathbb{B}(\mathcal{H})$ for $\lambda\in(0,1]$ \cite[Theorem 3.3]{zhou-glma-2023}.
	\end{proof}
	\begin{remark}
		The case $n=1$ is trivial. For $n\ge2$, the condition $(\Delta_{\lambda}(T))^n=\Delta_{\lambda}(T^n)$ does not generally imply quasinormality of $T^n$. Indeed, let $P=\frac{1}{2}\begin{bmatrix}
			1&1\\1&1
		\end{bmatrix}$ be the projection matrix, and consider $T=\begin{bmatrix}
			1&1\\0&0
		\end{bmatrix}$ 
		with the polar decomposition $T=U|T|$, where $U=\frac{\sqrt{2}}{2}\begin{bmatrix}
			1&1\\0&0
		\end{bmatrix}$ and $|T|=\sqrt{2}P$. A direct computation gives $\Delta_{\lambda}(T)=P$. Since $T^2=T$, we have
		$$\Delta_{\lambda}(T^n)=\Delta_{\lambda}(T)=P=P^2=P^n=(\Delta_{\lambda}(T))^n.$$However, $T^n=T$ is not quasinormal, since $TT^\ast T\ne T^\ast TT$.
	\end{remark}
	Let $n\in\{2,3,\cdots\}$. Taking the matrix in Remark \ref{squarezero}, one readily sees that there exists an $n$-power quasinormal operator $T$ such that $T^n$ is quasinormal, but $T$ itself is not quasinormal. We are now in a position to use Corollary \ref{erro} and the following theorem to characterize when an $n$-power quasinormal operator $T$ with quasinormal $T^n$ is actually quasinormal. We shall need the following lemma on paranormal operators.
	\begin{lemma}\label{parapro}
		Let $n\in\mathbb N$ and let $T$ be a paranormal operator. Then for every $x\in\mathcal H$,
		\begin{itemize}
			\item[(i)] $\|Tx\|^n \le \|T^nx\|\,\|x\|^{n-1}$.
			\item[(ii)] $\|T^nx\|^{n+1} \le \|T^{n+1}x\|^n\,\|x\|$.
			\item[(iii)] $\|Tx\|\,\|T^nx\| \le \|T^{n+1}x\|\,\|x\|$.
		\end{itemize}
	\end{lemma}
	
	\begin{proof}
		(i) The case $n=1$ is trivial. For $n\in\{2,3,\cdots\}$, since every paranormal operator is $n$-paranormal \cite[Theorem 1]{cho-jmr-2013}, (i) also holds.
		
		(ii) Since $T$ is paranormal, we have $\|Ty\|^2 \le \|T^2y\|\,\|y\|$ for all $y\in\mathcal H$. Setting $y = T^k x$ with $k\in\mathbb N_0$ and $x\in\mathcal H$, we get
		\begin{equation}\label{para3}
			\|T^{k+1}x\|^2 \le \|T^k x\|\,\|T^{k+2}x\|.
		\end{equation}
		If $T^n x = 0$, then the desired inequality is trivial. If $T^n x \ne 0$, then from \eqref{para3} we obtain
		\begin{equation}\label{paraineq1}
			\frac{\|Tx\|}{\|x\|}
			\le \frac{\|T^2x\|}{\|Tx\|}
			\le \frac{\|T^3x\|}{\|T^2x\|}
			\le \cdots
			\le \frac{\|T^n x\|}{\|T^{n-1}x\|}
			\le \frac{\|T^{n+1}x\|}{\|T^n x\|},
		\end{equation}
		and hence
		\[
		\frac{\|T^nx\|}{\|x\|}
		=
		\prod_{i=1}^{n} \frac{\|T^i x\|}{\|T^{i-1}x\|}
		\le
		\left(\frac{\|T^{n+1}x\|}{\|T^n x\|}\right)^n.
		\]
		Therefore the desired inequality follows.
		
		(iii) If $T^n x = 0$, inequality (iii) is trivial. If $T^n x \ne 0$, then from \eqref{paraineq1} we have
		\[
		\frac{\|Tx\|}{\|x\|} \le \frac{\|T^{n+1}x\|}{\|T^n x\|},
		\]
		that is, $\|Tx\|\,\|T^nx\| \le \|T^{n+1}x\|\,\|x\|$.
	\end{proof}
	
	\begin{theorem}\label{main1}
		Let $n\in\mathbb{N}$, and let $T\in\mathbb{B}(\mathcal{H})$ be an $n$-power quasinormal operator. If $T$ is paranormal, then $T$ is quasinormal.
	\end{theorem}
			\begin{proof}
				The case $n=1$ is trivial, since $1$-power quasinormal operators are quasinormal operators. Assume now that $n\in\{2,3,\dots\}$. Since $T$ is $n$-power quasinormal, we have $T^n T^*T = T^*T T^n$. By the continuous functional calculus applied to $T^*T$, we obtain
				\begin{equation}\label{paracom1}
					T^n |T| = |T| T^n.
				\end{equation}
				Taking adjoints in the above equality yields $(T^*)^n T^n |T| = |T| (T^*)^n T^n$, and hence
				\begin{equation}\label{paracom2}
					|T^n|\,|T| = |T|\,|T^n|.
				\end{equation}
Consider the spectral decomposition $|T| = \int_{[0,\infty)} \lambda\, dE(\lambda)$. Let $b>a\ge 0$. We claim that
				\begin{equation}\label{keyineq}
					a^n \|x\| \le \||T^n|x\| \le b^n \|x\|,\qquad x\in E([a,b])\mathcal H.
				\end{equation}
				Indeed, take $x\in E([a,b])\mathcal H$. Then
				\begin{equation}\label{paraine1}
					\||T|x\| \ge a\|x\|,
				\end{equation}
				and by \eqref{paracom1}, we have $T^nE([a,b])=E([a,b])T^n$, so $T^n x \in E([a,b])\mathcal H$. Hence
				\begin{equation}\label{paraine2}
					\|T^{n+1}x\|=\|TT^nx\|=\||T|T^nx\| \le b\|T^nx\|.
				\end{equation}
				If $x\ne0$, then by \eqref{paraine1} and Lemma \ref{parapro} (i),
				\[
				\||T^n|x\| = \|T^nx\| \ge \frac{\|Tx\|^n}{\|x\|^{n-1}} = \frac{\||T|x\|^n}{\|x\|^{n-1}} \ge a^n \|x\|.
				\]
				If $x=0$, the left-hand side of \eqref{keyineq} is trivial. Now suppose $T^n x \ne 0$. By Lemma \ref{parapro} (ii),
				\[
				\|T^nx\| \le \left(\frac{\|T^{n+1}x\|}{\|T^nx\|}\right)^n \|x\|,
				\]
				which, together with \eqref{paraine2}, gives
				\[
				\||T^n|x\| = \|T^nx\| \le b^n \|x\|.
				\]
				If $T^n x = 0$, the right-hand side of \eqref{keyineq} is trivial. Thus the claim follows. Moreover, by \eqref{paracom2}, the subspace $E([a,b])\mathcal H$ is invariant under $|T^n|$. Combining \eqref{keyineq} with the Löwner–Heinz inequality \cite{low-mz-1934}, we obtain on $E([a,b])\mathcal H$ that
				\begin{equation}\label{keyineq2}
					a^n E([a,b]) = a^n I \le |T^n|\big|_{E([a,b])\mathcal H} \le b^n I = b^n E([a,b]).
				\end{equation}
				
				If $\|T\|=0$, then $T=0$, and the theorem is trivial. Assume now that $\|T\|>0$. Take a partition $\pi: 0 = t_0 < t_1 < \cdots < t_m = \|T\|$. Define
				\[
				A_\pi = \sum_{i=2}^m t_{i-1}^n E((t_{i-1}, t_i]), \qquad
				B_\pi = t_1^n E([0,t_1]) + \sum_{i=2}^m t_i^n E((t_{i-1}, t_i]).
				\]
				From \eqref{keyineq2}, we have
				\[
				\sum_{i=2}^m t_{i-1}^n E((t_{i-1}, t_i])
				\le |T^n|\Big|_{\Big(\bigoplus_{i=2}^m E((t_{i-1}, t_i])\mathcal H\Big)\oplus E([0,t_{1}])\mathcal{H}}
				\le \sum_{i=2}^m t_i^n E((t_{i-1}, t_i])+t_{1}^nE([0,t_{1}]),
				\]
				and consequently
				\begin{equation}\label{exp1}
					A_\pi \le |T^n| \le B_\pi.
				\end{equation}
				Since for all $\lambda\in[0,\infty)$,
				\[
				\sum_{i=2}^m t_{i-1}^n \chi_{(t_{i-1}, t_i]}(\lambda)
				\le \lambda^n
				\le t_1^n \chi_{[0,t_1]}(\lambda)
				+ \sum_{i=2}^m t_i^n \chi_{(t_{i-1}, t_i]}(\lambda),
				\]
				where $\chi_\Delta$ denotes the characteristic function of a measurable set $\Delta\subseteq[0,\infty)$, we get
				\[
				A_\pi \le |T|^n \le B_\pi.
				\]
				Combining with \eqref{exp1}, we have
				$A_{\pi}-B_{\pi}\le|T^n|-|T^n|\le B_{\pi}-A_{\pi}$, and hence
				\[
				\bigl\||T^n| - |T|^n\bigr\| \le \|B_\pi - A_\pi\|.
				\]
				Moreover,
				\[
				\begin{aligned}
					\|B_\pi - A_\pi\|
					&= \Bigl\| t_1^n E([0,t_1]) + \sum_{i=2}^m (t_i^n - t_{i-1}^n)E((t_{i-1}, t_i]) \Bigr\| \\
					&\le \max_{1\le i\le m} (t_i^n - t_{i-1}^n) \\
					&= \max_i \Bigl( (t_i - t_{i-1}) \sum_{k=0}^{n-1} t_i^{n-1-k} t_{i-1}^k \Bigr) \\
					&\le n\|T\|^{n-1} \max_i (t_i - t_{i-1}).
				\end{aligned}
				\]
				Letting $|\pi| = \max\limits_i (t_i - t_{i-1}) \to 0$, we have
				\begin{equation}\label{keyineq3}
					|T^n| = |T|^n.
				\end{equation}
				Then by \eqref{paracom1}, we obtain
				\[
				T^n (T^n)^* T^n = T^n |T^n|^2
				= T^n |T|^n |T|^n
				= |T|^n |T|^n T^n
				= |T^n|^2 T^n
				= (T^n)^* T^n T^n,
				\]
				that is, $T^n$ is quasinormal.
				
				Set $S = T^n$. From \cite[Theorem 2.2 (ii)]{stankovic-afa-2025} and \eqref{keyineq3}, for every $m\in\mathbb{N}_{0}$,
				\[
				(S^*)^m S^m = (S^*S)^m = |T^n|^{2m} = |T|^{2nm}.
				\]
				Hence, for all $y\in\mathcal H$,
				\[
				\|T^{nm} y\| = \|S^m y\| = \||T|^{nm} y\|.
				\]
				Taking $y = Tx$ in the above equation, and using \eqref{paracom1}, we get
				\begin{equation}\label{normeq}
			\||T|^{nm} T x\|
			= \|T (T^n)^m x\|
			= \||T| (T^n)^m x\|
			= \|T^{nm} |T| x\|
			= \||T|^{nm+1} x\|.
				\end{equation}
				Define measures $\mu_x$ and $\nu_x$ on the Borel sets of $[0,\|T\|]$ by
				\[
				\mu_x(\Delta) = \langle E(\Delta) Tx, Tx\rangle,\qquad
				\nu_x(\Delta) = \langle E(\Delta) |T|x, |T|x\rangle.
				\]
				Then \eqref{normeq} gives
				\[
				\int_{[0,\|T\|]} t^{2nm}\, d\mu_x
				= \langle |T|^{2nm} T x, T x\rangle
				= \langle |T|^{2nm} |T|x, |T|x\rangle
				= \int_{[0,\|T\|]} t^{2nm}\, d\nu_x
				\]
				for every $m\in\mathbb{N}_{0}$. Thus, for every real polynomial $p$,
				\begin{equation}\label{approine}
					\int_{[0,\|T\|]} p(t^{2n})\, d\mu_x
					= \int_{[0,\|T\|]} p(t^{2n})\, d\nu_x.
				\end{equation}
Let $C([0,\|T\|])$ be the set of continuous real-valued functions on $[0,\|T\|]$. Now take $f\in C([0,\|T\|])$ and define $g(s)=f(s^{1/(2n)})$ on $[0,\|T\|^{2n}]$. By the Weierstrass approximation theorem, there exists a sequence of real polynomials $\{p_k\}$ converging uniformly to $g$ on $[0,\|T\|^{2n}]$. Putting $s=t^{2n}$, the polynomials $\{p_k(t^{2n})\}$ converge uniformly to $g(t^{2n})=f(t)$ on $[0,\|T\|]$. By \eqref{approine},
				\[
				\int_{[0,\|T\|]} f\, d\mu_x
				= \int_{[0,\|T\|]} f\, d\nu_x
				\]
				for every $f\in C([0,\|T\|])$. The uniqueness in the Riesz-Markov representation theorem yields $\mu_x = \nu_x$, and hence for every Borel set $\Delta$,
				\begin{equation}\label{keyineq4}
					\|E(\Delta) T x\| = \|E(\Delta) |T| x\|.
				\end{equation}
				
				Now take $u\in E(\Delta)\mathcal H$. By \eqref{keyineq4},
				\[
				\begin{aligned}
					\|(1-E(\Delta)) T u\|^2
					&= \|T u\|^2 - \|E(\Delta) T u\|^2 \\
					&= \||T| u\|^2 - \|E(\Delta) T u\|^2 \\
					&= \|E(\Delta) |T| u\|^2 - \|E(\Delta) T u\|^2 \\
					&= 0,
				\end{aligned}
				\]
				so $T u \in E(\Delta)\mathcal H$. For $v\in (1-E(\Delta))\mathcal H$, by \eqref{keyineq4},
				\[
				\|E(\Delta) T v\| = \|E(\Delta) |T| v\| = \||T| E(\Delta) v\| = 0,
				\]
				so $T v \in (1-E(\Delta))\mathcal H$. Therefore, $E(\Delta)\mathcal H$ reduces $T$ for every Borel set $\Delta$. Consequently, for any $x\in\mathcal H$,
				\[
				E(\Delta) T x
				= E(\Delta) (T E(\Delta) x + T(1-E(\Delta)) x)
				= E(\Delta) T E(\Delta) x
				= T E(\Delta) x,
				\]
				so $E(\Delta) T = T E(\Delta)$ for every Borel set $\Delta\subseteq[0,\|T\|]$. Since $E$ is the spectral measure of $|T|$, we get $|T|T = T|T|$. Hence
				\[
				T^* T T = |T|^2 T = T |T|^2 = T T^* T,
				\]
				which means that $T$ is quasinormal. The proof is complete.
			\end{proof}
	
	\begin{corollary}\label{powerquasi}
		Let $\lambda\in(0,1]$, $n\in\{2,3,4,\cdots\}$ and let $T\in\mathbb{B}(\mathcal{H})$ be an $n$-power quasinormal operator. If $T^n$ is quasinormal, then the following are equivalent:
		\begin{itemize}
			\item[(i)] $T$ is quasinormal.
			\item[(ii)] $\|Tx\|^2\le\|T^2x\|\|x\|+\|(\Delta_{\lambda}(T))^n-\Delta_{\lambda}(T^n)\|\|x\|^2$ for all $x\in\mathcal{H}$.
			\item[(iii)]
			$|T|^2\le|T^2|+\|(\Delta_{\lambda}(T))^n-\Delta_{\lambda}(T^n)\|I$.
		\end{itemize}
	\end{corollary}
	\begin{proof}
		We shall prove the implications in the order
	$${\rm (iii)}\Longrightarrow{\rm (ii)}
	\Longrightarrow{\rm (i)}
	\Longrightarrow{\rm (iii)}.$$		
		
	(iii) $\Rightarrow$ (ii) requires no hypotheses. Indeed, from (iii) we have 	$$\langle |T|^2x,x\rangle\le\langle|T^2|x,x\rangle+\langle\|(\Delta_{\lambda}(T))^n-\Delta_{\lambda}(T^n)\|x,x\rangle$$ for $x\in\mathcal{H}$. Applying the Cauchy-Schwarz inequality, which yields $$\langle |T^2|x,x \rangle\le\||T^2|x\|\|x\|=\|T^2x\|\|x\|,$$
	and using the identity $\langle|T|^2x,x\rangle=\||T|x\|^2=\|T\|^2$, we arrive at (ii).
	
	(ii) $\Rightarrow$ (i). Since $T^n$	is quasinormal, we have $\Delta_{\lambda}(T^n)=T^n$. By Theorem \ref{char}, we also have $(\Delta_{\lambda}(T))^n=T^n$. Therefore, from (ii), we obtain $\|Tx\|^2\le\|T^2x\|\|x\|$ for all $x\in\mathcal{H}$, which means that $T$ is paranormal. Then by Theorem \ref{main1}, we get (i).
	
	(i) $\Rightarrow$ (iii). From \eqref{normalclass}, $T$ is a class A operator, so $|T|^2\le|T^2|$, and hence (iii) follows immediately.
	\end{proof} 
	
The following example shows that if $T^n$ is quasinormal and satisfies (iii) (and hence, as can be seen from the proof of Corollary \ref{powerquasi}, also satisfies (ii)), but $T$ itself is not $n$-power quasinormal, then one cannot conclude that $T$ is quasinormal.		
	\begin{example}
		Let $n\in\{3,4,\cdots\}$. Let $R = \begin{bmatrix}
			S & I \\
			0 & \omega S
		\end{bmatrix}$ be the operator on $l^2(\mathbb N_0) \oplus l^2(\mathbb N_0)$, where $\omega = e^{2\pi i/n}$ and $S = \operatorname{shift}(1,1,1,\dots)$. A direct induction gives
		$
		R^n =
		\begin{bmatrix}
			S^n & S^{n-1} \sum_{j=0}^{n-1} \omega^j \\
			0 & \omega^n S^n
		\end{bmatrix}
		=
		\begin{bmatrix}
			S^n & 0 \\
			0 & S^n
		\end{bmatrix}$,
		since $\sum_{j=0}^{n-1} \omega^j = 0$ and $\omega^n = 1$. Moreover, since $S^*S = I$, we have $(R^n)^*R^n = I$, and hence $R^n$ is quasinormal. On the other hand,
		\[
		R^n R^* R - R^* R R^n =
		\begin{bmatrix}
			\times & S^n S^* - S^* S^n \\
			\times & \times
		\end{bmatrix}
		=
		\begin{bmatrix}
			\times & -S^{n-1} P_{\langle e_0\rangle} \\
			\times & \times
		\end{bmatrix}
		\neq 0,
		\]
		so $R$ is not $n$-power quasinormal, and hence not quasinormal.
		
		We claim that $d\triangleq\|(\Delta_{\lambda}(R))^n - \Delta_{\lambda}(R^n)\| > 0$. Indeed, if $d = 0$, then 
		\[
		(\Delta_{\lambda}(R))^n = \Delta_{\lambda}(R^n) = R^n,
		\]
		which would imply, by Theorem \ref{char}, that $R$ is $n$-power quasinormal, a contradiction.	Now we can choose $\alpha > 0$ such that
		\[
		\alpha^{n-2} d \ge \|R\|^2.
		\]
		
		Set $T = \alpha R$. Then $T^n$ is quasinormal, while $T$ is not $n$-power quasinormal (and hence not quasinormal). Moreover, by the homogeneity of the $\lambda$-Aluthge transform, we have
		\[
		\|(\Delta_{\lambda}(T))^n - \Delta_{\lambda}(T^n)\|
		= \|\alpha^n ((\Delta_{\lambda}(R))^n - \Delta_{\lambda}(R^n))\|
		= \alpha^n d,
		\]
		and consequently
		\[
		\begin{aligned}
			|T|^2
			&\le \||T|^2\| I
			= \|T\|^2 I
			= \alpha^2 \|R\|^2 I \\
			&\le \alpha^n dI
			= \|(\Delta_{\lambda}(T))^n - \Delta_{\lambda}(T^n)\|I \\
			&\le |T^2| + \|(\Delta_{\lambda}(T))^n - \Delta_{\lambda}(T^n)\|I,
		\end{aligned}
		\]
		that is, $T$ satisfies condition (iii).
\end{example}
In what follows, we construct an operator $S$ using the direct sum of operators, showing that $S$ is $n$-power quasinormal and satisfies condition (ii) of Corollary \ref{powerquasi}, while $S^n$ is not quasinormal and (iii) fails to hold. Consequently, $S$ is not quasinormal either.
		\begin{example}
			Let $n\in\{3,4,\cdots\}$. Take the operator
			\[
			T = \begin{bmatrix}
				W & 0 \\
				\sqrt{1-c^2} P_{\langle e_0 \rangle} & 0
			\end{bmatrix}
			\]
			from Example \ref{pnnotimplynq}. It is easy to verify that \(W^n\) has the polar decomposition \(W^n = U|W^n|\), where $|W^n| e_k =
			\begin{cases}
				c e_0 & k = 0, \\
				e_k & k \ge 1,
			\end{cases}$
		and
			$U e_k = e_{k+n} \quad (k \ge 0).$
			Thus
			\[
			\Delta_\lambda(W^n) e_k
			= |W^n|^\lambda U |W^n|^{1-\lambda} e_k
			=
			\begin{cases}
				c^{1-\lambda} e_n & k = 0, \\
				e_{n+k} & k \ge 1.
			\end{cases}
			\]
Moreover, by \eqref{poma}, the polar decomposition of \(T^n\) is $T^n =\begin{bmatrix}
				U & 0 \\
				0 & 0
			\end{bmatrix}
			\begin{bmatrix}
				|W^n| & 0 \\
				0 & 0
			\end{bmatrix}$,
so $$\Delta_\lambda(T^n) =
			\begin{bmatrix}
				\Delta_\lambda(W^n) & 0 \\
				0 & 0
			\end{bmatrix}.$$
		Example \ref{pnnotimplynq} gives $T$ $n$-power quasinormal, so Theorem \ref{char} yields $$(\Delta_\lambda(T))^n = T^n =
			\begin{bmatrix}
				W^n & 0 \\
				0 & 0
			\end{bmatrix}.$$
			Combining this with $
			(W^n - \Delta_\lambda(W^n)) e_k =
			\begin{cases}
				(c - c^{1-\lambda}) e_n & k = 0, \\
				0 & k \ge 1,
			\end{cases}$
			we obtain
			\begin{equation}\label{aluthgeeq2}
			\|(\Delta_\lambda(T))^n - \Delta_\lambda(T^n)\|
			= \|W^n - \Delta_\lambda(W^n)\|
			= c^{1-\lambda} - c.
			\end{equation}
			
	Now take any $x = \begin{pmatrix} x_1 \\ x_2 \end{pmatrix}
			\in l^2(\mathbb N_0) \oplus l^2(\mathbb N_0)$.
			From \eqref{aluthgema}, we have 
			\begin{equation}\label{aluthgeeq1}
			\|Tx\| = \||T|x\| = \|x_1\|.
			\end{equation}
			 By \eqref{poma} and \eqref{poma1},
			\[
			\begin{aligned}
				\|T^2x\|^2
				&= \|W^2 x_1\|^2\\&= \left\| W^2 \sum_{k=0}^\infty \langle x_1, e_k \rangle e_k \right\|^2 \\
				&= \left\| c \langle x_1, e_0 \rangle e_2 + \sum_{k=1}^\infty \langle x_1, e_k \rangle e_{k+2} \right\|^2 \\
				&= c^2 \|\langle x_1, e_0 \rangle e_2\|^2
				+ \sum_{k=1}^\infty \|\langle x_1, e_k \rangle e_{k+2}\|^2 \\
				&\ge c^2 \sum_{k=0}^\infty \|\langle x_1, e_k \rangle e_{k+2}\|^2
				= c^2 \|x_1\|^2.
			\end{aligned}
			\]
			Hence, from \eqref{aluthgeeq1} together with the above equation, we get $$\|Tx\|^2 - \|T^2x\|\,\|x\|
			\le \|x_1\|^2 - c\,\|x_1\|\,\|x\|
			\le (1-c)\|x\|^2.$$
			Let $\alpha = \left( \frac{1-c}{c^{1-\lambda} - c} \right)^{1/(n-2)}$. Replacing \(x\) by \(\alpha x\) in the above inequality gives
			\begin{equation}\label{alutheeq5}
				\|\alpha T x\|^2
				\le \|(\alpha T)^2 x\|\,\|x\|
				+ \alpha^2 (1-c)\|x\|^2.
			\end{equation}
			
			On the other hand, take the \(3 \times 3\) matrix $
			M = \begin{pmatrix}
				0 & 2 & -1 \\
				0 & 0 & 2 \\
				0 & 0 & 0
			\end{pmatrix}$.
		Clearly \(M^3 = 0\), so \(M\) is \(n\)-power quasinormal, and hence $(\Delta_\lambda(M))^n = M^n = 0 = \Delta_\lambda(M^n)$, so 
		\begin{equation}\label{aluthgeq3}
			\|(\Delta_\lambda(M))^n - \Delta_\lambda(M^n)\| = 0.
		\end{equation}
	 For any unit vector $y = \begin{pmatrix} y_1 \\ y_2 \\ y_3 \end{pmatrix} \in \mathbb C^3$, we have
			\[
			\begin{aligned}
				\|My\|^2 - \|M^2 y\|
				&= |2y_2 - y_3|^2 + 4|y_3|^2 - 4|y_3| \\
				&\le 4|y_2|^2 + 4|y_2||y_3| + 5|y_3|^2 - 4|y_3| \\
				&\le 4(1 - |y_3|^2) + 4|y_3|\sqrt{1 - |y_3|^2} + 5|y_3|^2 - 4|y_3| \\
				&= 4 + |y_3|^2 + 4|y_3|\sqrt{1 - |y_3|^2} - 4|y_3| \\
				&\le 4 + |y_3|^2 + 4|y_3|\left(1 - \frac{|y_3|^2}{2}\right) - 4|y_3| \\
				&= 4 + |y_3|^2 - 2|y_3|^3 \le \frac{109}{27}.
			\end{aligned}
			\]
			Thus, for every \(x \in \mathbb C^3\),
			\[
			\|Mx\|^2 \le \|M^2 x\|\,\|x\| + \frac{109}{27}\|x\|^2.
			\]
Let $\beta = \sqrt{\frac{27}{109}}\, \alpha \sqrt{1-c}$. Replacing \(x\) by \(\beta x\) in the above inequality gives
\begin{equation}\label{aluthgeeq6}
\|\beta M x\|^2
\le \|(\beta M)^2 x\|\,\|x\|
+ \frac{109}{27}\beta^2 \|x\|^2
= \|(\beta M)^2 x\|\,\|x\|
+ \alpha^2(1-c)\|x\|^2.
\end{equation}

Write \(\mathcal K^2 = l^2(\mathbb N_0) \oplus l^2(\mathbb N_0)\). Now consider the operator $$S = \alpha T \oplus \beta M$$ on \(\mathcal K^2 \oplus \mathbb C^3\). Since \(T\) and \(M\) are \(n\)-power quasinormal and \(T^n\) is not quasinormal, we have that \(S\) is \(n\)-power quasinormal and \(S^n\) is not quasinormal. Moreover,
			\[
			\begin{aligned}
				&\|(\Delta_\lambda(S))^n - \Delta_\lambda(S^n)\| \\
				&= \|(\Delta_\lambda(\alpha T \oplus \beta M))^n - \Delta_\lambda(\alpha^n T^n \oplus \beta^n M^n)\| \\
				&= \|(\Delta_\lambda(\alpha T))^n \oplus (\Delta_\lambda(\beta M))^n
				- \Delta_\lambda(\alpha^n T^n) \oplus \Delta_\lambda(\beta^n M^n)\| \\
				&= \max \bigl\{ \|(\Delta_\lambda(\alpha T))^n - \Delta_\lambda(\alpha^n T^n)\|,
				\|(\Delta_\lambda(\beta M))^n - \Delta_\lambda(\beta^n M^n)\| \bigr\} \\
				&= \max \bigl\{ \|\alpha^n ((\Delta_\lambda(T))^n - \Delta_\lambda(T^n))\|,
				\|\beta^n ((\Delta_\lambda(M))^n - \Delta_\lambda(M^n))\| \bigr\}.
			\end{aligned}
			\]
			Combining this with \eqref{aluthgeeq2} and \eqref{aluthgeq3}, we obtain
			\[
			\|(\Delta_\lambda(S))^n - \Delta_\lambda(S^n)\|
			= \alpha^n (c^{1-\lambda} - c)
			= \alpha^2 (1-c).
			\]
Now take any $x = \begin{pmatrix} u \\ v \end{pmatrix}
			\in \mathcal K^2 \oplus \mathbb C^3$. Then \eqref{alutheeq5} and \eqref{aluthgeeq6} yield
			\[
			\begin{aligned}
				\|Sx\|^2
				&= \|\alpha T u\|^2 + \|\beta M v\|^2 \\
				&\le \|(\alpha T)^2 u\|\,\|u\|
				+ \|(\beta M)^2 v\|\,\|v\|
				+ \alpha^2(1-c)(\|u\|^2 + \|v\|^2) \\
				&\le \sqrt{\|(\alpha T)^2 u\|^2 + \|(\beta M)^2 v\|^2}
				\sqrt{\|u\|^2 + \|v\|^2}
				+ \alpha^2(1-c)\|x\|^2 \\
				&= \|S^2 x\|\,\|x\|
				+ \alpha^2(1-c)\|x\|^2 \\
				&= \|S^2 x\|\,\|x\|
				+ \|(\Delta_\lambda(S))^n - \Delta_\lambda(S^n)\|\,\|x\|^2,
			\end{aligned}
			\]
			where the second inequality follows from the Cauchy–Schwarz inequality. Hence $S$ satisfies condition (ii). Finally, take $\widetilde v = \begin{pmatrix} 0 \\ v \end{pmatrix}
			\in \mathcal K^2 \oplus \mathbb C^3$, where $v = \frac{\sqrt{5}}{5}
			\begin{pmatrix} 0 \\ 2 \\ -1 \end{pmatrix}$. A direct computation gives
			\[
			\begin{aligned}
				&\langle |S|^2 \widetilde v, \widetilde v \rangle
				- \langle |S^2| \widetilde v, \widetilde v \rangle \\
				&= \langle |\beta M|^2 v, v \rangle
				- \langle |(\beta M)^2| v, v \rangle \\
				&= \beta^2 \langle Mv, Mv \rangle
				- \beta^2
				\left\langle
				\begin{pmatrix} 0 & 0 & 0 \\ 0 & 0 & 0 \\ 0 & 0 & 4 \end{pmatrix}
				v, v
				\right\rangle \\
				&= 5\beta^2
				= \frac{135}{109} \alpha^2(1-c)
				> \alpha^2(1-c)
				= \|(\Delta_\lambda(S))^n - \Delta_\lambda(S^n)\|\langle\widetilde{v},\widetilde{v}\rangle.
			\end{aligned}
			\]
			Hence \(S\) does not satisfy condition (iii).
		\end{example}
		\begin{remark}
			The operator \(S\) in the example is not paranormal either; hence the class of operators satisfying condition (ii) is larger than the class of paranormal operators. Indeed, using the above notation, a direct computation gives
			\[
			\|\beta M v\|^2 = \frac{29}{5}\beta^2
			> \frac{4\sqrt{5}}{5}\beta^2
			= \|(\beta M)^2 v\|\,\|v\|,
			\]
			and therefore $\|S\widetilde v\|^2
			= \|\beta M v\|^2
			> \|(\beta M)^2 v\|\,\|v\|
			= \|S^2 \widetilde v\|\,\|\widetilde v\|$, so \(S\) is not paranormal.
		\end{remark}
		
		From Theorem \ref{main1}, it follows that if condition (ii) in Corollary \ref{powerquasi} is strengthened to $\|Tx\|^2\le\|T^2x\|\|x\|$ for all $x\in\mathcal{H}$, i.e., $T$ is paranormal, then the quasinormality of $T$ can be inferred even without assuming that $T^n$ is quasinormal. A parallel question is whether, under the strengthened condition that $T$ is paranormal, the hypothesis that $T$ is $n$-power quasinormal can be omitted, while retaining the assumption that $T^n$is quasinormal. This is exactly the question posed by Stanković and Kubrusly (Question \ref{question1}). We give a positive answer to this question at the end of the paper. The following key lemma is first established.
		
	\begin{lemma}\label{paralem}
		Let $n\in\mathbb N$ and let $T\in\mathbb B(\mathcal H)$ be paranormal. If $T^n$ is quasinormal, then $T|T^n|=|T^n|T$.
	\end{lemma}
	\begin{proof}
		The cases $n=1$ or $\|T^n\|=0$ are trivial. Assume now that $n\ge2$ and $\|T^n\|>0$. Our aim is to derive estimates for \eqref{conkeyine2} and \eqref{conparaine2}, and consequently to estimate \eqref{conkeyine3}, which will complete the proof. 
		
		Since $T^n$ is quasinormal, it follows from \cite[Theorem 2.2 (ii)]{stankovic-afa-2025} that
		\[
		(T^*)^{nm}T^{nm} = ((T^*)^n T^n)^m = |T^n|^{2m}
		\]
		for every $m\in\mathbb N_{0}$, and hence for every $x\in\mathcal H$,
		\begin{equation}\label{conkeyineq}
			\|T^{nm}x\| = \|(T^n)^m x\| = \||T^n|^m x\|.
		\end{equation}
		Moreover, by Lemma \ref{parapro} (iii), for every $m\in\mathbb N_{0}$,
		\[
		\|T^{nm+1}x\|\,\|T^{nm+n}x\|
		\le \|T^{nm+n+1}x\|\,\|T^{nm}x\|.
		\]
		Using \eqref{conkeyineq}, we obtain
		\begin{equation}\label{paralem8}
			\||T^n|^m T x\|\,\||T^n|^{m+1} x\|
			\le \||T^n|^{m+1} T x\|\,\||T^n|^m x\|.
		\end{equation}
		
		Let $N$ be a positive integer and set $h=\frac{\|T^n\|}{N}$. Let $E$ be the spectral measure of $|T^n|$. Since $T^n$ is quasinormal, i.e., $T^n|T^n|=|T^n|T^n$, we have $T^n E(\Delta)=E(\Delta)T^n$ for every Borel set $\Delta\subseteq[0,\|T^n\|]$; that is, $E(\Delta)$ is invariant under $T^n$. Take a nonzero vector $x\in E((kh,(k+1)h])\mathcal H$ with $k\in\{1,2,\dots,N-1\}$. By \eqref{conkeyineq},
		\[
		\|T^{nm+n}x\| = \||T^n|^{m+1}x\|
		= \bigl\||T^n|^{m+1}\big|_{E((kh,(k+1)h])\mathcal H} x\bigr\|
		\ge (kh)^{m+1}\|x\| \ne 0.
		\]
		Therefore, from \eqref{conkeyineq} and the above equation, we have $\||T^n|^m T^j x\| = \|T^{nm+j}x\|\ne 0$ for every $j\in\{0,1,\dots,n\}$. Repeatedly applying \eqref{paralem8} gives
		\begin{equation}\label{paralem2}
			\begin{aligned}
				\frac{\||T^n|^{m+1}x\|}{\||T^n|^m x\|}
				&\le \frac{\||T^n|^{m+1}Tx\|}{\||T^n|^m T x\|}
				\le \frac{\||T^n|^{m+1}T^2x\|}{\||T^n|^m T^2 x\|}
				\le \cdots \\
				&\le \frac{\||T^n|^{m+1}T^n x\|}{\||T^n|^m T^n x\|}
				= \frac{\|T^n |T^n|^{m+1}x\|}{\|T^n |T^n|^m x\|}
				= \frac{\||T^n|^{m+2}x\|}{\||T^n|^{m+1}x\|}.
			\end{aligned}
		\end{equation}
	Furthermore,
		\begin{equation}\label{paralem3}
			\begin{aligned}
				\frac{\||T^n|^{m+2}x\|^2}{\||T^n|^{m+1}x\|^2}
				- \frac{\||T^n|^{m+1}x\|^2}{\||T^n|^m x\|^2}
				&= \frac{\langle |T^n|^4 y_m, y_m\rangle}{\langle |T^n|^2 y_m, y_m\rangle}
				- \langle |T^n|^2 y_m, y_m\rangle \\
				&\le \frac{(1+2k)^2 h^4}{4\langle |T^n|^2 y_m, y_m\rangle}
				\le \frac{(1+2k)^2 h^4}{4k^2 h^2}
				\le\frac{9}{4}h^2,
			\end{aligned}
		\end{equation}
		where $y_m = \frac{|T^n|^m x}{\||T^n|^m x\|}$, and the first inequality follows from inequality (1.3) in \cite{zuo-jmi-2013}. Taking $m=0,1$ in \eqref{paralem2} and \eqref{paralem3}, together with the fact that $x\in E((kh,(k+1)h])\mathcal H$ is nonzero, we get
		\begin{equation}\label{paralem4}
			kh \le \frac{\||T^n|x\|}{\|x\|}
			\le \frac{\||T^n|Tx\|}{\|Tx\|}
			\le \frac{\||T^n|^2x\|}{\||T^n|x\|}
			\le \frac{\||T^n|^2Tx\|}{\||T^n|Tx\|}
			\le \frac{\||T^n|^3x\|}{\||T^n|^2x\|}
			\le (k+1)h,
		\end{equation}
		and
		\begin{equation}\label{paralem5}
			\begin{aligned}
				\frac{\||T^n|^3x\|^2}{\||T^n|^2x\|^2}
				- \frac{\||T^n|x\|^2}{\|x\|^2}
				&=
				\left( \frac{\||T^n|^3x\|^2}{\||T^n|^2x\|^2}
				- \frac{\||T^n|^2x\|^2}{\||T^n|x\|^2} \right)
				+
				\left( \frac{\||T^n|^2x\|^2}{\||T^n|x\|^2}
				- \frac{\||T^n|x\|^2}{\|x\|^2} \right) \\
				&\le \frac{9}{4}h^2 + \frac{9}{4}h^2
				= \frac{9}{2}h^2.
			\end{aligned}
		\end{equation}
From \eqref{paralem4} and \eqref{paralem5}, we have
$$\frac{\||T^n|^2Tx\|^2}{\||T^n|Tx\|^2}
- \frac{\||T^n|Tx\|^2}{\|Tx\|^2}
\le \frac{9}{2}h^2.$$
Write $\widetilde{x}= \left( |T^n| - \frac{\||T^n|Tx\|}{\|Tx\|} \right) Tx$. Then a direct computation gives
\begin{equation}\label{keyineq6}
\begin{aligned}
	\Big\|\Big(|T^n| + \frac{\||T^n|Tx\|}{\|Tx\|} I\Big)\widetilde{x}\Big\|^2&=\left\| \left( |T^n|^2 - \frac{\||T^n|Tx\|^2}{\|Tx\|^2} \right) Tx \right\|^2 \\
	&= \||T^n|^2 Tx\|^2
	- 2\frac{\||T^n|Tx\|^2}{\|Tx\|^2}\||T^n|Tx\|^2
	+ \frac{\||T^n|Tx\|^4}{\|Tx\|^4}\|Tx\|^2 \\
	&= \frac{\||T^n|Tx\|^2}{\|Tx\|^2}
	\left( \frac{\||T^n|^2Tx\|^2}{\||T^n|Tx\|^2}
	- \frac{\||T^n|Tx\|^2}{\|Tx\|^2} \right) \|Tx\|^2\\&\le\frac{\||T^n|Tx\|^2}{\|Tx\|^2}\cdot\frac{9}{2}h^2\|Tx\|^2.
\end{aligned}
\end{equation}
On the other hand, $\frac{\||T^n|Tx\|}{\|Tx\|} I
\le |T^n| + \frac{\||T^n|Tx\|}{\|Tx\|} I$
gives $\frac{\||T^n|Tx\|^2}{\|Tx\|^2} I
\le \left( |T^n| + \frac{\||T^n|Tx\|}{\|Tx\|} I \right)^2$, so 
$$	\Big\|\frac{\||T^n|Tx\|}{\|Tx\|}\widetilde{x}\Big\|\le\Big\|\Big(|T^n| + \frac{\||T^n|Tx\|}{\|Tx\|} I\Big)\widetilde{x}\Big\|.$$
Combining the above with \eqref{keyineq6}, we get
$$\Big\|\frac{\||T^n|Tx\|}{\|Tx\|}\widetilde{x}\Big\|\le\frac{\||T^n|Tx\|}{\|Tx\|} \cdot \frac{3}{\sqrt{2}} h \|Tx\|,$$
and hence
		\begin{equation}\label{paralem7}
			\left\| \left( |T^n| - \frac{\||T^n|Tx\|}{\|Tx\|} \right) Tx \right\|
			\le \frac{3}{\sqrt{2}} h \|Tx\|.
		\end{equation}
Moreover, from \eqref{paralem4}, we have
		\[
		\left\| \left( |T^n| - \frac{\||T^n|Tx\|}{\|Tx\|} \right) \Big|_{E((kh,(k+1)h])\mathcal H} \right\|
		\le\max\left\{ \frac{\||T^n|Tx\|}{\|Tx\|} - kh,\ (k+1)h - \frac{\||T^n|Tx\|}{\|Tx\|} \right\}
		\le h.
		\]
		Hence, by \eqref{paralem7}, for every nonzero $x\in E((kh,(k+1)h])\mathcal H$ with $k\in\{1,2,\dots,N-1\}$, we have
		\begin{equation}\label{conkeyine2}
			\begin{aligned}
				\|(|T^n|T - T|T^n|)x\|
				&\le
				\left\| \left( |T^n| - \frac{\||T^n|Tx\|}{\|Tx\|} \right) Tx \right\|
				+
				\left\| \frac{\||T^n|Tx\|}{\|Tx\|} Tx - T|T^n|x \right\| \\
				&\le
				\frac{3}{\sqrt{2}} h \|T\| \|x\|
				+
				\|T\|
				\left\|
				\left( |T^n| - \frac{\||T^n|Tx\|}{\|Tx\|} \right)
				\Big|_{E((kh,(k+1)h])\mathcal H}
				x
				\right\| \\
				&\le
				\frac{3}{\sqrt{2}} h \|T\| \|x\|
				+
				\|T\|
				\left\|
				\left( |T^n| - \frac{\||T^n|Tx\|}{\|Tx\|} \right)
				\Big|_{E((kh,(k+1)h])\mathcal H}
				\right\|
				\|x\| \\
				&\le
				\frac{3}{\sqrt{2}} h \|T\| \|x\|
				+ h \|T\| \|x\|.
			\end{aligned}
		\end{equation}
		
		Let $\varepsilon>0$. Define	$\widetilde T = E([h+\varepsilon,\|T^n\|]) T E([0,h])$, $A = T^n\big|_{E([h+\varepsilon,\|T^n\|])\mathcal H}$ and $B = T^n\big|_{E([0,h])\mathcal H}$. Then $A\widetilde T = E([h+\varepsilon,\|T^n\|]) T^{n+1} E([0,h]) = \widetilde T B$. Hence, for every $m\in\mathbb N$,
		\begin{equation}\label{lemine1}
			A^m \widetilde T = \widetilde T B^m.
		\end{equation}
Take $x\in E([0,h])\mathcal H$. By \eqref{conkeyineq} and \eqref{lemine1},
		\[
		\begin{aligned}
			(h+\varepsilon)^m \|\widetilde T x\|
			&\le \||T^n|^m \widetilde T x\|
			= \|T^{nm} \widetilde T x\|
			= \|A^m \widetilde T x\|
			= \|\widetilde T B^m x\| \\
			&= \|\widetilde T T^{nm} x\|
			\le \|\widetilde T\| \, \||T^n|^m x\|
			\le h^m \|\widetilde T\| \, \|x\|.
		\end{aligned}
		\]
		Thus $\|\widetilde T x\| \le \frac{h^m}{(h+\varepsilon)^m} \|\widetilde T\| \|x\|$. Letting $m\to\infty$, we get $\widetilde T x = E([h+\varepsilon,\|T^n\|]) T E([0,h]) x = 0$. Since $E([h+\varepsilon,\|T^n\|])$ converges to $E((h,\|T^n\|])$ in the strong operator topology as $\varepsilon\to0$, we have
		\[
		E((h,\|T^n\|]) T E([0,h]) x
		= \lim_{\varepsilon\to0} E([h+\varepsilon,\|T^n\|]) T E([0,h]) x
		= 0.
		\]
		Therefore, $E([0,h])\mathcal H$ is invariant under $T$. Consequently, for every $x\in E([0,h])\mathcal H$,
		\begin{equation}\label{conparaine2}
			\begin{aligned}
				\|(|T^n|T - T|T^n|)x\|
				&\le
				\left\| \left( |T^n| - \frac{h}{2} \right) Tx \right\|
				+
				\left\| \frac{h}{2} Tx - T|T^n|x \right\| \\
				&\le
				\left\| \left( |T^n| - \frac{h}{2} \right) \Big|_{E([0,h])\mathcal H} \right\| \|Tx\|
				+
				\|T\| \left\| \left( \frac{h}{2} - |T^n| \right) \Big|_{E([0,h])\mathcal H} \right\| \|x\| \\
				&\le
				2 \cdot \frac{h}{2} \|T\| \|x\|
				= h \|T\| \|x\|.
			\end{aligned}
		\end{equation}
		
		Finally, take any $x\in\mathcal H$. Then $x = E([0,h])x + \sum_{k=1}^{N-1} E((kh,(k+1)h])x$. Since $h=\frac{\|T^n\|}{N}$, by \eqref{conkeyine2} and \eqref{conparaine2} we obtain
		\begin{equation}\label{conkeyine3}
				\begin{aligned}
				\|(|T^n|T - T|T^n|)x\|
				&\le
				\|(|T^n|T - T|T^n|)E([0,h])x\|
				+
				\sum_{k=1}^{N-1}
				\|(|T^n|T - T|T^n|)E((kh,(k+1)h])x\| \\
				&\le
				\left( \frac{3}{\sqrt{2}} + 1 \right) h \|T\|
				\left( \|E([0,h])x\| + \sum_{k=1}^{N-1} \|E((kh,(k+1)h])x\| \right) \\
				&\le
				\left( \frac{3}{\sqrt{2}} + 1 \right) h \|T\| \sqrt{N}
				\sqrt{
					\|E([0,h])x\|^2 + \sum_{k=1}^{N-1} \|E((kh,(k+1)h])x\|^2
				} \\
				&=
				\left( \frac{3}{\sqrt{2}} + 1 \right) h \|T\| \sqrt{N} \|x\| \\
				&=
				\frac{1}{\sqrt{N}}
				\left( \frac{3}{\sqrt{2}} + 1 \right) \|T^n\| \|T\| \|x\|,
			\end{aligned}
		\end{equation}
		where the third inequality follows from the Cauchy–Schwarz inequality. Letting $N\to\infty$, we obtain $T|T^n| = |T^n|T$. The proof is complete.
	\end{proof}
	
	We now answer Question \ref{question1}.
	
	\begin{theorem}
	If $T\in\mathbb{B}(\mathcal{H})$ is paranormal and $T^n$ is quasinormal for some $n\in\mathbb{N}$, then $T$ is quasinormal.
	\end{theorem}
	\begin{proof}	
		The case $n=1$ is trivial. If $\|T^n\|=0$, then by Lemma \ref{parapro} (i), we have $Tx=0$ for every $x\in\mathcal H$, so $T=0$, and the result follows. Assume now that $n\ge2$ and $\|T^n\|>0$. Take a positive integer $N$. Let $h = \frac{\|T^n\|}{N}$ and let $E$ be the spectral measure of $|T^n|$. For any $x\in E((kh,(k+1)h])\mathcal H$, $k=1,2,\dots,N-1$, Lemma \ref{parapro} (i) gives
		\begin{equation}\label{keypara1}
			\|Tx\|^n \le \|T^nx\|\,\|x\|^{n-1}
			= \||T^n|x\|\,\|x\|^{n-1}
			\le (k+1)h\,\|x\|^n.
		\end{equation}
		
		By Lemma \ref{paralem}, we have $T^*|T^n| = |T^n|T^*$, and hence $E([0,h])\mathcal H$ and $E((kh,(k+1)h])\mathcal H$ ($k=1,2,\dots,N-1$) are reducing subspaces for $T$. From \eqref{keypara1}, we obtain
		\[
		\|T\Big|_{E((kh,(k+1)h])\mathcal H}\| \le ((k+1)h)^\frac{1}{n}.
		\]
		Thus, for $x\in E((kh,(k+1)h])$,
		$$\begin{aligned}
			\|T^nx\|
			&= \|(T\Big|_{E((kh,(k+1)h])\mathcal H})^n x\|
			\\&\le \|T\Big|_{E((kh,(k+1)h])\mathcal H}\|^{n-1}
			\|T\Big|_{E((kh,(k+1)h])\mathcal H} x\|
			\le ((k+1)h)^\frac{n-1}{n} \|Tx\|.
		\end{aligned}$$
Also, since $kh\|x\| \le \||T^n|x\| \le \|T^nx\|$, we have $kh\|x\| \le ((k+1)h)^{(n-1)/n} \|Tx\|$, that is,
		\[
		\|Tx\| \ge k(k+1)^{(1-n)/n} h^{1/n} \|x\|.
		\]	
Hence, from the above and \eqref{keypara1}, we obtain
		\[
		k^2(k+1)^{(2-2n)/n} h^{2/n}
		\le T^*T\Big|_{E((kh,(k+1)h])\mathcal H}
		\le (k+1)^{2/n} h^{2/n}.
		\]
		Moreover,
		\[
		k^2(k+1)^{(2-2n)/n} h^{2/n}
		\le (kh)^{2/n}
		\le |T^n|^{2/n}\Big|_{E((kh,(k+1)h])\mathcal H}
		\le ((k+1)h)^{2/n}.
		\]
		Therefore, for $k\ge1$ and $n\ge2$,
		\[
		\begin{aligned}
			-2h^{2/n}
			&< k^2(k+1)^{(2-2n)/n} h^{2/n} - ((k+1)h)^{2/n} \\
			&\le (T^*T - |T^n|^{2/n})\Big|_{E((kh,(k+1)h])\mathcal H} \\
			&\le ((k+1)h)^{2/n} - k^2 h^{2/n} (k+1)^{(2-2n)/n} \\
			&= (2k+1)(k+1)^{2/n-2} h^{2/n}
			< 2h^{2/n}.
		\end{aligned}
		\]
		Thus, for every $k\in\{1,2,\dots,N-1\}$,
		\begin{equation}\label{queq1}
			\|(T^*T - |T^n|^{2/n}) E((kh,(k+1)h])\| < 2h^{2/n}.
		\end{equation}
		
		On the other hand, by Lemma \ref{parapro} (i), for $x\in E([0,h])\mathcal H$,
		\[
		\|Tx\|^n \le \|T^nx\|\,\|x\|^{n-1}
		= \||T^n|x\|\,\|x\|^{n-1}
		\le h\|x\|^n.
		\]
		Hence $0 \le T^*T\Big|_{E([0,h])\mathcal H} \le h^{2/n}I$, and since $0 \le |T^n|^{2/n}\Big|_{E([0,h])\mathcal H} \le h^{2/n}I$, we have $$-h^{2/n} \le (T^*T - |T^n|^{2/n})\Big|_{E([0,h])\mathcal H}
		\le T^*T\Big|_{E([0,h])\mathcal H}
		\le h^{2/n}.$$ Thus
		\begin{equation}\label{queq2}
			\|(T^*T - |T^n|^{2/n}) E([0,h])\| \le h^{2/n}.
		\end{equation}
		
		Now, since $E([0,h])\mathcal H$ and $E((kh,(k+1)h])\mathcal H$ ($k=1,2,\dots,N-1$) are invariant subspaces for $T^*T - |T^n|^{2/n}$, we conclude from \eqref{queq1} and \eqref{queq2} that
		\[
		\begin{aligned}
			&\|T^*T - |T^n|^{2/n}\|
			= \left\| (T^*T - |T^n|^{2/n}) \left( E([0,h]) + \sum_{k=1}^{N-1} E((kh,(k+1)h]) \right) \right\| \\
			&= \max \Bigg\{ \|(T^*T - |T^n|^{2/n}) E([0,h])\|, \max_{1\le k\le N-1} \|(T^*T - |T^n|^{2/n}) E((kh,(k+1)h])\| \Bigg\} \\
			&< 2h^{2/n} = 2\left( \frac{\|T^n\|}{N} \right)^{2/n}.
		\end{aligned}
		\]
		Letting $N\to\infty$, we get $T^*T = |T^n|^{2/n}$, i.e., $|T|^2 = |T^n|^{2/n}$. By the uniqueness of the positive square root, $|T| = |T^n|^{1/n}$, and hence $|T|^n = |T^n|$. Combining this with Lemma \ref{paralem}, we obtain
		\[
		T|T|^n = T|T^n| = |T^n|T = |T|^n T.
		\]
		By the functional calculus of $|T|^n$, this gives $T|T| = |T|T$, which yields that $T$ is quasinormal.
	\end{proof}

	\bibliographystyle{amsplain}
	
\end{document}